  \documentclass[10pt, a4paper ] {article}
  \pdfoutput=1
  
\usepackage{amsmath,amsthm,amssymb}
\usepackage{graphicx,color}
\usepackage{dsfont}
\usepackage{float}
\usepackage{caption}
\usepackage{subcaption}
\usepackage[]{algorithm2e}

\usepackage{makecell}
\usepackage{gnuplottex}
\usepackage{pgfplots}

\usepackage{tikz} % To generate the plot from csv
\usepackage{pgfplots}
  
\usepackage{accents}

  \newcommand{\overbar}[1]{\mkern 1.5mu\overline{\mkern-1.5mu#1\mkern-1.5mu}\mkern 1.5mu}

  \newcommand{\bproof}{\begin{proof}}
      \newcommand{\eproof}{\end{proof}}

    \newcommand{\ben}{\begin{equation}}
    \newcommand{\een}{\end{equation}}

    \newcommand{\benn}{\begin{equation*}}
    \newcommand{\eenn}{\end{equation*}}

    \newcommand{\R}{\mathbf{R}}
    \newcommand{\N}{\mathbf{N}}

    \newcommand{\divv}{\operatorname{div}}

    \newcommand{\vb}{\mathbf{v}}

    \newcommand{\lb}{[\![}
    \newcommand{\rb}{]\!]}
    \newcommand{\Hrep}{\mathcal H(\mathcal X,\R^d)}
    \newcommand{\VN}{\mathcal V_N(\mathcal X,\R^d)}

    \newcommand{\gb}{\mathbf{g}}

    \newcommand{\dt}{{\frac{d}{dt}}}

    \newcommand{\tr}{\operatorname{tr}}

    \newcommand{\Sb}{\mathbf{S}}
    \newcommand{\Sf}{\mathfrak{S}}

    \newcommand{\nf}{\nu}
    
     \newcommand{\ac}{\accentset{\circ}}
    
    \theoremstyle{definition}
    \newtheorem{thrm}{Theorem}[section]
    
    \newtheorem{theorem}[thrm]{Theorem}
    \newtheorem{lemma}[thrm]{Lemma}
    \newtheorem{corollary}[thrm]{Corollary}
    \newtheorem{remark}[thrm]{Remark}
    
    \newtheorem{definition}[thrm]{Definition}
    
    \newtheorem{assumption}[thrm]{Assumption}

    \newtheorem{example}[thrm]{Example}

    \numberwithin{equation}{section}

\providecommand{\keywords}[1]{\textbf{\textit{Keywords: }} #1}
\providecommand{\AMS}[1]{\textbf{\textit{AMS classification: }} #1}
   \DeclareMathOperator{\si}{sinc}

\begin{document}

\newpage

\title{Reproducing kernel Hilbert spaces and variable metric algorithms in PDE constrained shape optimisation}
\author{M. Eigel\thanks{ Weierstrass-Institute,
        Mohrenstr. 39,
        10117 Berlin Germany,
        martin.eigel@wias-berlin.de } \;          %  \\}        \and
    and K. Sturm \thanks{
        Universit\"at Duisburg-Essen,
        Fakult\"at f\"ur Mathematik,
        Thea-Leymann-Str. 9,
        45127 Essen,
        Germany,
        kevin.sturm@uni-due.de}
}
\date{} 
    \maketitle
    
    \begin{abstract}
In this paper we investigate and compare different gradient algorithms 
designed for the domain expression of the shape derivative. Our main focus is 
to examine the usefulness of kernel reproducing Hilbert spaces for PDE 
constrained
shape optimisation problems. We show that radial kernels 
provide convenient formulas for the shape gradient that can be efficiently 
used in numerical simulations. 
The shape gradients associated with radial kernels depend on a so called smoothing parameter that allows a 
smoothness  adjustment of the shape during the optimisation process. 
Besides, this smoothing parameter can be used to modify the
movement of the shape. 
The theoretical findings are verified in a number of numerical 
experiments. 
    \end{abstract}
    
    
    % REQUIRED
    \AMS{
        35J15; 46E22; 49Q10; 49K20; 49K40
    }

\section{Introduction}
\label{sec:introduction}
Optimal shape design questions naturally arise from problems
in the engineering sciences and industrial applications. For instance, it plays 
an important role in
aircraft design, electrical impedance tomography, 
cantilever designs, inductor coil design and many more. 
The main objective of shape optimisation is to minimise a certain cost/shape 
function depending on one or many design variables.
A great challenge, relevant for applications, 
is to find fast and efficient algorithms providing as output (locally) optimal shapes. One may define first and second order methods by means of the so called shape 
derivative.
\\[0.2cm]
A central result of shape optimisation constitutes the \emph{structure theorem} 
for shape functions defined on open or closed subsets of the Euclidean space.
As a consequence of the structure theorem we can identify, in smooth situations, the shape derivative with a distribution on the boundary only depending on 
normal perturbations. In many applications this distribution can 
be written as boundary integral which is referred to as \emph{boundary expression}. 
If the shape is not smooth enough one still can conclude that the shape derivative is concentrated on the 
boundary, but it may not necessarily be a distribution on the boundary anymore. 
However, for many application problems, a weaker form of the 
shape derivative is usually available. This form can be referred to as \emph{volume/domain expression} or \emph{distributed shape derivative} and
it can be written in a convenient tensor form as 
detailed in \cite{lauraindistributed}. 
\\[0.2cm]
By definition the \emph{shape gradient} of the shape derivative depends 
on the choice of the Hilbert space and inner product. It is nothing but the Riesz representation 
of the shape derivative in this Hilbert space. Using the boundary expression of the 
shape derivative has the advantage that it allows to resort to boundary spaces.
For PDE constrained optimal design problems, many gradient-type algorithms using the boundary 
expression in conjunction with boundary spaces have been proposed by employing various explicit 
parametrisations such as B\'ezier splines, B-splines, NURBS; see e.g. \cite{MR2887931,SturmHoemHint13,MR2299620,MR1939127,ls_algo_sensi,MR1911658}. While the boundary 
expression gives a relatively easy formula of the shape derivative, it is
not the first choice from the numerical point of view as recently pointed out in 
\cite{MR3348199,lauraindistributed,MR2642680}. By definition the shape gradient depends
on the choice of the Hilbert space where the shape derivative is represented.
While some choices using $H^1$ metrics and finite elements have
successfully been used \cite{lauraindistributed,MR3436555}, the question arises if there are
better Hilbert spaces and metrics that are more controllable. At best, one might want to change the 
metric during the optimisation process in order to escape stationary points 
that are no global minima.  
\\[0.2cm]
Reproducing kernel Hilbert spaces (RKHSs) were introduced in the beginning of 
the 19th century. They play a crucial role in polynomial approximation and machine learning.
We refer to \cite{MR2131724} for an introduction to  RKHS and their 
application to scattered data approximation. 
RKHS can be extended to vector valued reproducing kernel Hilbert spaces (vvRKHS).
As  shown in  \cite{MR2656312} they can also efficiently be used to solve
diffeomorhpic matching problems. A specific property of vvRKHS is that the point 
evaluation on them is a continuous linear mapping. Conversely, the continuity
of the evaluation mapping in a Hilbert space implies that it is a vvRKHS. The continuity of the 
evaluation mapping is also necessary to build complete metric groups of diffeomorphisms
as demonstrated in \cite[Chap. 4]{MR2731611}. This shows that there is a close 
relation between RKHS and shape design problems.  Therefore, it seems natural to  
combine and examine results from RKHS theory with problems from 
PDE constrained shape optimisation.  
\\[0.2cm]
In this paper we examine the usefulness of reproducing kernel Hilbert spaces 
in the context of PDE constrained shape optimisation problems. 
We combine the generic tensor form of the domain expression of the shape 
derivative with reproducing kernel Hilbert space methods. We provide ready to use 
explicit formulas for the shape gradient in these kernel spaces and compare them with previously used ones.
Moreover, we study radial kernels that allow us to construct flows that
can efficiently detect stationary points. Our theoretical results 
are verified by several numerical experiments. 
\\[0.3cm]
\subsubsection*{Structure of the paper}
\hspace*{0.1em}\\
In Section~\ref{sec:preliminary}, we review basic results from shape calculus and 
recall the recently introduced tensor representation of the shape derivative. 
We recall the definition of the gradient of the shape derivative and define
descent directions. 

In Section~\ref{sec:kernel}, we introduce the theory of reproducing 
kernel Hilbert spaces and relate them to the shape derivative. Explicit 
formulas of gradients in general reproducing kernel Hilbert spaces are obtained 
that can be readily used in 
numerical algorithms. The general results are specialised to 
radial kernels and the relation. At the end of the section, different
approaches to obtain descent directions are proposed and compared.

In Section~\ref{sec:model_problems}, a transmission problem together with a 
tracking-type cost function is studied. 
We give a detailed description of the 
discretisation of the PDE and of the shape derivative.  In 
a general tensor setting we compare the discrete domain and boundary expression. 

In Section~\ref{sec:experiments}, the previously introduced methods are 
tested in a number of numerical experiments. 

\section{Preliminaries}\label{sec:preliminary}
\setcounter{equation}{0}
In this section, we recall some basics from shape calculus. For an in-depth 
treatment we refer the reader to the 
monographs \cite{MR2731611,SokoZol,HenPie05}. Numerous examples of 
PDE constrained shape functions and their shape derivatives can be found in \cite{sturm2015shape}. 

\subsection{Flow of vector fields and shape derivative}
Subsequently, let $D\subset \R^d$, $d\ge 1$, be an open and bounded set.
Given a function $X\in \ac C^{0,1}(\overbar{D},\R^d)$, we denote by $\Phi_t$ the flow of $X$ 
(short $X$-flow) given by $\Phi_t(x_0) := x(t)$, where $x(\cdot)$ solves
\ben\label{eq:flow}
x'(t) = X(x(t)) \qquad\textsf{ in }\qquad
(0,\tau], \; x(0)=x_0.	
\een  
The space $\ac C^{0,1}(\overbar{D}, \R^d)$ comprises all bounded and Lipschitz 
continuous functions on $\overbar{D}$ vanishing on $\partial D$. 
Note that by the chain rule $\partial\Phi^{-1}(t,\Phi(t,x))=(\partial \Phi(t,x))^{-1}$ which we will often write as
\begin{equation}
(\partial (\Phi^{-1}_t))\circ \Phi_t=(\partial \Phi_t)^{-1}=:\partial \Phi^{-1}_t.
\end{equation}
By $\ac C^k(\overbar D,\R^d)$ we denote the subspace of $k$-times continuously 
differentiable functions on $\overbar D$ vanishing on $\partial \Omega$. 
For open and bounded sets $\Omega\subset \R^d$ and for all 
finite integers $p\ge 1$ and $k\ge 1$, we define the Sobolev space 
$\ac{W}^k_p(\Omega,\R^d) = \overline{C^\infty_c(\Omega,\R^d)}^{\|\cdot\|_{W^k_p(\Omega,\R^d)}}$. 
Moreover, for all  open and bounded sets $\Omega\subset \R^d$ with Lipschitz 
boundary $\partial \Omega$, we define 
$W^k_p(\Omega,\R^d) = \overline{C^\infty(\overbar{\Omega},\R^d)}^{\|\cdot\|_{W^k_p(\Omega,\R^d)}}.$ 
As usual in case $p=2$ we set $H^k(\Omega,\R^d):=W^k_2(\Omega,\R^d)$ and $\ac{H}^k(\Omega,\R^d):=\ac{W}^k_2(\Omega,\R^d)$.

\begin{definition}\label{def1} 
    Let $D\subset \R^d$ be an open set and $J:\Xi\subset \wp(D) \rightarrow \R$ a shape function defined on subsets
    of $D$. We denote by $\wp(D)$ the power set of $D$. Let $\Omega\in \Xi$ and $X \in C^k(\overbar D,\R^d)$, $k\ge 1$, be such that 
    $\Phi_t(\Omega) \in \Xi$ for all 
    $t >0$ sufficiently small.   
    Then the  \emph{Eulerian semi-derivative} of $J$ at $\Omega$ in direction $X$ is defined by
    \ben
    dJ(\Omega)(X):= \lim_{t \searrow 0}\frac{J(\Phi_t(\Omega))-J(\Omega)}{t} .
    \een
    \begin{itemize}
        \item[(i)] The function $J
        $ is said to be \textit{shape differentiable} at $\Omega$ 
        if for some $k\ge 1$ the Eulerian semi-derivative $dJ(\Omega)(X)$ 
        exists for all $X \in \ac C^k(\overbar D,\R^d)$ and 
        $ X     \mapsto dJ(\Omega)(X) $ 
        is linear and continuous on $\ac C^k(\overbar D,\R^d)$.
        \item[(ii)]  The smallest integer  $k\ge 0$ for which 
        $X \mapsto dJ(\Omega)(X)
        $ is continuous with respect to the 
        $ C^k(\overbar D,\R^d)$-topology is called the order of $dJ(\Omega)$. 
    \end{itemize}
\end{definition}

An important result of shape optimisation constitutes the so-called 
\emph{structure theorem} that gives a characterisation of shape derivatives in 
open or closed sets $\Omega$. When the boundary of $\Omega$ admits 
some regularity and the shape derivative is a distribution of certain order,
then the structure theorem tells us that the derivative depends only on normal perturbations.
\begin{theorem}\label{thm:structure}
    Let $J:\Xi\subset \wp(D) \rightarrow \R$ be a shape function and 
    $\Omega \subset \Xi$ open or closed with $C^{k+1}$ boundary $\Gamma$. 
    Suppose that $J$ is shape differentiable at $\Omega$ and that it is 
    of order $k$. Then there exists a scalar distribution 
    $g(\Omega)\in (C^k(\Gamma))^*$ such that
    \ben\label{eq:structure_g}
    dJ(\Omega)(X)=\langle g(\Omega), X\cdot \nu\rangle_{(C^k(\Gamma))^*,C^k(\Gamma)}\quad \textsf{ for all } X\in C^k_c(D, \R^d). 	
    \een
\end{theorem}
For a proof of the previous theorem we refer the reader to \cite{MR2731611}.

\subsection{Tensor representation of the shape derivative}
In the recent work \cite{lauraindistributed}, a generic tensor form 
of the shape derivative was proposed. Our further investigation 
benefits from this tensor form as it allows us to obtain convenient formulas of 
shape gradients and it helps to distinguish the discretised and
non-discretised shape derivative. 
\begin{definition}\label{def:tensor}
    Let $\Omega \in \wp(D)$ be a set with $C^1$ boundary. 
    Assume $J$ is shape differentiable at $\Omega$ 
    and that its shape derivative $dJ(\Omega)$ is of order $k=1$.  
    We say that the shape derivative of $J$ admits a \emph{tensor representation} at $\Omega$. If
    there exist tensors $\Sb_1\in  L_1(D, \R^{d,d})$, $\Sb_0\in L_1(D,\R^d)$  and 
    $\Sf_1\in L_1(\partial \Omega; \R^{d,d})$, 
    $\Sf_0\in L_1(\partial \Omega, \R^d)$
    such that for $X\in \ac C^1(\overbar D,\R^d)$,
    \ben\label{ea:volume_from}
    dJ(\Omega)(X) =\int_D \Sb_1: \partial X + \Sb_0\cdot 
    X \,dx + \int_{\partial \Omega} \Sf_1: \partial_\Gamma 
    X + \Sf_0 \cdot X\, ds,
    \een
    where $\partial_\Gamma X : = \partial X -(\partial X n)\otimes n$ is the 
    tangential derivative of $X$ along $\partial\Omega$. 
\end{definition}
\begin{remark}
    \begin{itemize}
        \item The functions $\Sb_0,\Sb_1$ and $\Sf_0,\Sf_1$ depend on the domain 
        $\Omega$.  When necessary, we explicitly express the dependence 
        of $\Sb_0,\Sb_1$ and $\Sf_0,\Sf_1$ on $\Omega$ by writing 
        $\Sb_0(\Omega),\Sb_1(\Omega)$ and $\Sf_0(\Omega),\Sf_1(\Omega)$, respectively. 
        \item The tensor representation \eqref{ea:volume_from} is not 
        unique.  In fact in Example~\ref{ex:extension} below  
        we show that one can obtain different tensor 
        representations of the same shape derivative by choosing 
        different inner products under the assumption that $dJ(\Omega)(\cdot)$ belongs
        to some Hilbert space. 
        \item When $dJ(\Omega)
        $ admits a tensor representation of order one, then $dJ(\Omega)(X)$ 
        naturally extends to vector fields $X\in W^1_\infty(D,\R^d)$ 
        by means of the right hand side of \eqref{ea:volume_from}.  
    \end{itemize}
\end{remark}
\begin{example}\label{ex:shape_func_f}
    As an example we consider
    an open subset $D\subset \R^d$ and the shape function $J(\Omega):=\int_D f_\Omega\, dx$ with $f_\Omega := f_1 \chi_\Omega + 
    f_2 \chi_{D\setminus \Omega}$,  $f_1,f_2\in C^1_c(\R^d,\R^d)$.  Then 
    $J$ is shape differentiable in all measurable subsets $\Omega\subset D$ and 
    the shape derivative in direction $X\in \ac C^1(\overbar D,\R^d)$ is given by 
    \ben
    dJ(\Omega)(X) = \int_D \Sb_1(\Omega):\partial X + \Sb_0(\Omega)\cdot X\;dx,	
    \een
    where 
    $$ \Sb_1(\Omega) := f_\Omega I, \quad \Sb_0(\Omega) := \chi_\Omega \nabla f_1 + \chi_{D\setminus \Omega} \nabla f_2.	$$
    Hence, in this case $\Sf_0(\Omega)=0$ and $\Sf_1(\Omega)=0$. 
    We refer the reader to \cite{lauraindistributed} for more examples of 
    shape derivatives admitting a tensor representation. 
\end{example}
\begin{example}\label{ex:extension}
    Let $J$ be a shape function such that 
    $dJ(\Omega)(X)$ is well-defined for all $X$ in  $\ac C^1(\overbar D,\R^d)$ and assume that it can be extended to a functional 
    $\widetilde{dJ}(\Omega)$ on $\ac{H}^1(D,\R^d)$.  Then 
    Riesz representation theorem states that there is a unique $\gb_\Omega \in \ac{H}^1(D, \R^d)$ such that
    \ben\label{eq:inner}
    \widetilde{dJ}(\Omega)(X) = \int_D \partial \gb_\Omega : \partial X + \gb_\Omega \cdot X \, dx \quad \textsf{ for all } X \in \ac{H}^1(D, \R^d).	
    \een
    Restricting $\widetilde{dJ}(\Omega)$ to smooth vector fields in $\ac C^1(\overbar D,\R^d)$, we recover formula 
    \eqref{ea:volume_from} with $\Sb_1:=\partial g_\Omega, \Sb_0:=g_\Omega$, 
    $\Sf_0=0$ and $\Sf_1=0$. 
    Of course instead of using the inner product on the right hand side of 
    \eqref{eq:inner} one could alternatively solve: find 
    $\tilde g_\Omega\in \ac H^1(D,\R^d)$ so that
    $$ \widetilde{dJ}(\Omega)(X) = \int_D \partial \tilde\gb_\Omega : \partial X \, dx \quad \textsf{ for all } X \in \ac{H}^1(D, \R^d),	$$
    then we get a different tensor form with $\Sb_1:=\partial \tilde g_\Omega, \Sb_0:=0$, 
    $\Sf_0=0$ and $\Sf_1=0$. 
\end{example}
Example~\ref{ex:extension} suggests to investigate shape functions with shape derivatives of 
order $k=1$ having a tensor representation of the form
\begin{equation}\label{eq:domain} 
dJ(\Omega)(X) = \int_D \Sb_1 : \partial X + \Sb_0\cdot X \, dx, \quad 
X\in \ac C^1(\overbar D,\R^d),
\end{equation} 
where $\Sb_1 \in L_1(D, \R^{d,d})$ and $\Sb_0 \in L_1(D, \R^d)$.
%For a function $\Sb\in L_1(D,\R^{d,d})$ the \emph{deviator} is defined (in a 
%pointwise a.e. sense) by
%$ \Sb^D := \Sb - \frac{1}{d} \tr(\Sb) I.   $
%By definition, $\tr(\Sb^D)=0$ a.e. on $D$. Hence, equation \eqref{eq:domain} can
%be rewritten as
%\begin{equation}
%	dJ(\Omega)(X) = \int_D \Sb_1^D : \partial X\; dx  + \int_D \frac{1}{d} \tr(\Sb_1) 
%	\divv(X) + \Sb_0\cdot X \, dx  
% \end{equation}  
% for $X\in \ac C^1(D,\R^d)$. 

Under the assumption $\Sb_1\in W^1_1(D,\R^{d,d})$ we readily recover (cf. \cite[Prop. 3.3]{lauraindistributed}) the so-called boundary expression from 
\eqref{eq:domain}  
\begin{equation}\label{eq:boundary}
dJ(\Omega)(X) = \int_{\partial \Omega} \lb \Sb_1\nu\cdot \nu\rb \, X\cdot \nf 
\, ds, \quad X\in \ac C^1(\overbar D,\R^d),
\end{equation}
where $\lb\Sb_1\nu\cdot \nu\rb:= (\Sb_1^+-\Sb_1^-) \nf \cdot \nf$ denotes the 
jump of $\Sb_1$ across $\Gamma$ and  $\nu$ is the outward-pointing unit vector 
field along $\Gamma$. This formula is in accordance 
with~\eqref{eq:structure_g} of Theorem~\ref{thm:structure}. The 
$\pm$ indicates the restriction of the function to $\Omega^\pm$, respectively, for example 
$\Sb_1^\pm := (\Sb_1)|_{\Omega^\pm}$ with $\Omega^+:=\Omega$ and 
$\Omega^-:=D\setminus \Omega$. Here 
the involved tensor fields additionally satisfy the conservation equations
\ben \label{eq:equvilibrium_strong}
\begin{split}
    -\divv(\Sb_1^+) + \Sb_0^+ &= 0 \quad \textsf{ in } \Omega \\
    -\divv(\Sb_1^-) + \Sb_0^- &= 0 \quad \textsf{ in } D\setminus\overbar{\Omega}.
\end{split}
\een
Note that if the boundary $\partial \Omega$ is irregular, say $\Omega$ is 
merely bounded and open, formula \eqref{eq:domain} and \eqref{eq:boundary} 
are not equivalent. In fact, in this case  \eqref{eq:boundary} is in general 
not well-defined. 

It is important to notice that after discretisation the equality of \eqref{eq:domain} and \eqref{eq:boundary} 
breaks down as pointed out in \cite{MR3348199}.
For a numerical and theoretical 
comparison of the boundary and domain expression we refer to \cite{MR2642680,MR3348199,lauraindistributed}.

\subsection{Shape gradients and descent directions}\label{sec:gradient-flow}
Let $D\subset \R^d$ be an open set and $\Xi\subset \wp(D)$ a subset of the powerset of $D$.
Consider a shape function $J:\Xi \rightarrow \R$ that is shape differentiable at $\Omega\in \Xi$. 
Suppose there is a Hilbert space $\mathcal H(\mathcal X,\R^d)$ of functions from 
$\mathcal X$ into $\R^d$ and assume $dJ(\Omega) \in \mathcal H(\mathcal X,\R^d)^*$. 
\begin{definition}\label{def:H_gradient}
    \begin{itemize}
        \item[(i)] 
        The gradient of $J$ at $\Omega$ with respect to the space 
        $\mathcal{H}(\mathcal X,\R^d)$ and the inner product $(\cdot, \cdot )_{\mathcal{H}(\mathcal X,\R^d)}$,
        denoted $\nabla J(\Omega)$, is defined by
        %\marginpar{\tiny denoted $\nabla J(\Omega)$???}
        \begin{equation}\label{eq:H_gradient}
        dJ(\Omega)(X)=( \nabla J(\Omega), X)_{\mathcal{H}(\mathcal 
            X,\R^d)} \; \textsf{ for all } X\in \mathcal{H}(\mathcal X,\R^d).
        \end{equation}
        We also call $\nabla J(\Omega)$ the $\mathcal H(\mathcal X,\R^d)$-gradient of $J$ 
        at $\Omega$.
    \end{itemize}
    \begin{remark}
        The Hilbert space $\mathcal H(\mathcal X,\R^d)$ may be equipped with 
        different scalar products $(\cdot, \cdot )_{\mathcal{H}(\mathcal 
            X,\R^d)}$ yielding the same topology on 
        $\mathcal H(\mathcal X,\R^d)$. 
    \end{remark}
\end{definition}
\begin{example}
    Consider the shape function $J
    $ from Example~\ref{ex:shape_func_f}. Let $\Omega\in D$ be open and set $\mathcal X:=D$. Then it is easy to see that $dJ(\Omega)$ belongs 
    to $\mathcal H(D,\R^d):=H^1_0(D,\R^d)$. The gradient $ \nabla J(\Omega)$ 
    with respect to the metric
    $ (\varphi,\psi)_{\ac{H}^1} := \int_D \partial \varphi : \partial \psi \; dx  $  
    is then defined as the solution of
    $$ ( \nabla J(\Omega),  X)_{\ac H^1} = \int_D \Sb_1(\Omega):\partial X + \Sb_0(\Omega) \cdot X \, dx \quad \textsf{ for all } X \in H^1_0(D,\R^d).	$$	
\end{example}
As shown by the next lemma, the negative gradient is nothing 
but the steepest descent direction for the shape derivative.
\begin{lemma}
    Let $J$ and $\mathcal H = \mathcal H(\mathcal X,\R^d)$ be as in 
    Definition~\ref{def:H_gradient} and suppose $dJ(\Omega)\ne 0$.
    Then there exists a unique $\gb_\Omega\in \mathcal H(\mathcal X,\R^d)$ with norm equal to one, satisfying
    $$ \min_{\substack{\vb\in \mathcal H\\ \|\vb\|_{\mathcal H}=1}} dJ(\Omega)(\vb) =  dJ(\Omega)(\gb_\Omega), 	$$
    where $\gb_\Omega$ is given by
    $ \gb_\Omega := - \nabla J(\Omega)/\|\nabla J(\Omega)\|_{\mathcal H(\mathcal X,\R^d)}. $
\end{lemma}
\begin{proof}
    By Cauchy Schwarz's inequality, we get 
    $ (\nabla J(\Omega),-X)_{\mathcal H(\mathcal X,\R^d)} \le \|\nabla J(\Omega)\|_{\mathcal H(\mathcal X,\R^d)}$
    for all $X \in \mathcal H(\mathcal X,\R^d)$ with $\|X\|_{\mathcal H(\mathcal X,\R^d)}=1$, 
    which is equivalent to
    $   ( \nabla J(\Omega), -\gb_\Omega)_{\mathcal H(\mathcal X,\R^d)} \le (\nabla J(\Omega), X)_{\mathcal H(\mathcal X,\R^d)} $
    for all $X\in \mathcal H(\mathcal X,\R^d)$ with $\|X\|_{\mathcal H(\mathcal X,\R^d)}=1$. This proves
    existence and also uniqueness of the minimiser since the Cauchy-Schwarz inequality is an equality if and only if the vectors are colinear.
\end{proof}

\begin{remark}			
    Suppose that for all $\Omega\in \Xi$ there is a Hilbert space $\mathcal H(\Omega,\R^d)$ of functions from $\Omega$ into $\R^d$.  
    Consider a shape function $J:\Xi \rightarrow \R$, $\Omega\in \Xi$ and assume 
    that $dJ(\Omega) \in \mathcal H(\Omega,\R^d)^*$.
    Then formally the gradient of $J$ is a mapping
    $ \nabla J: \Xi \rightarrow \bigcup_{\Omega\in \Xi} \mathcal H(\Omega,\R^d) $
    satisfying
    $ \nabla J(\Omega)\in \mathcal H(\Omega,\R^d)	$ for all $\Omega\in \Xi$. 
    If we regards $\mathcal H(\Omega,\R^d)$ as the tangent space of $\Xi$ in the point 
    $\Omega$, we can interpret $\nabla J$ as a vector field. Of course at 
    this stage $\Xi$ has no differentiable structure turning it into a manifold.  
    However, there are several possibilities to do this. One way is to introduce 
    spaces of shapes via curves cf. \cite{MR2201275}, but there are several other 
    ways to put some structure on $\Xi$; cf.  \cite[Chapter 3-7]{MR2731611}.
\end{remark}

\begin{definition}
    We call a vector field $X\in \mathcal H(D,\R^d)
    $ descent direction for $J$ at $\Omega \in \Xi$ if $dJ(\Omega)(X)$ exists 
    and $dJ(\Omega)(X) < 0$. 
\end{definition}

\section{Reproducing kernel Hilbert spaces and the shape derivative}\label{sec:kernel}
In this section we recall the definition 
of reproducing kernel Hilbert spaces (RKHSs) and their basic properties.  
We give some examples of kernels that can be used in PDE constrained 
shape optimisation. 
The aim is now to introduce certain Hilbert spaces $\mathcal H$ namely 
\emph{reproducing kernel Hilbert spaces} that allow
explicit representations of the gradient $\nabla J(\Omega)$ of shape functions. 

\subsection{Definition and basic properties of reproducing kernels}
Let $\mathcal X\subset \overbar{D}$ be an arbitrary and given set. 
We denote by 
$\mathcal{H} = \mathcal H(\mathcal X;\R^d)$ a 
real Hilbert space of vector valued functions $f:\mathcal X\rightarrow \R^d$ that will 
be specified later on. In case $d=1$ we set 
$\mathcal H:= \mathcal H(\mathcal X) := \mathcal H(\mathcal X, \R)$. 
\begin{definition}\label{def:kernel_scalar}	 
    \begin{itemize}
        \item[(a)] A function 
        $k:\mathcal X \times \mathcal X \rightarrow \R$ is called positive (semi)-definite and 
        symmetric \emph{scalar kernel} if
        \begin{itemize}
            \item[$(a_1)$] $\forall x,y\in \mathcal X$, $k(x,y)=k(y,x)$
            \item[$(a_2)$] for arbitrary pairwise distinct points 
            $\{x_1,\ldots, x_N\} \subset \mathcal X$, $N\ge 1$, the matrix $k_{ij}:=k(x_i,x_j)$ 
            is positive (semi)-definite, i.e., for all $\alpha \in \R^N\setminus \{0\}$,
            $$ \sum_{i,j=1}^N  \alpha_i \alpha_j k_{ij}  \ge(>)   0. $$
        \end{itemize}
        \item[(b)] A kernel $k$ is called radial scalar kernel, if there 
        exists a function $\gamma:\R\rightarrow \R$ such that for all 
        $x,y\in \mathcal X$, $k(x,y)= \gamma(|x-y|)$. 
        \item[(c)] A function $f:\mathcal X\rightarrow \R$ is called 
        positive (semi)-definite, if $k(x,y):=f(x-y)$ is a 
        positive (semi)-definite kernel. 
        \item[(d)] A function 
        $k:\mathcal X\times \mathcal X \rightarrow \R$ is 
        called \emph{scalar reproducing kernel} for $\mathcal H(\mathcal X)$ if 
        \begin{itemize}
            \item[$(d_1)$] for all 
            $ x\in \mathcal X, \;k(x,\cdot)\in \mathcal H(\mathcal X)$
            \item[$(d_2)$] for all $ f\in \mathcal H(\mathcal X)$ and for all $ x\in \mathcal X$,
            \ben\label{eq:rep_vec}
            (k(x,\cdot), f(\cdot))_{\mathcal H(\mathcal X)} = f(x).	
            \een
        \end{itemize}
        In this case we call $\mathcal H(\mathcal X)$ \emph{reproducing kernel Hilbert space} 
        with kernel $k$. 
    \end{itemize}
\end{definition}
It is readily seen that a reproducing kernel 
$k:\mathcal X\times \mathcal X\rightarrow \R$ is symmetric. Indeed, using the 
reproducing property $(d_2)$, 
we obtain for all 
$x,y\in \mathcal X$, 
\ben\label{eq:kernel_x_y}
(k(x, \cdot), k(y,\cdot))_{\mathcal H(\mathcal X)} = k(x,y).	
\een
Hence, scalar reproducing kernels are always symmetric.
But they are also positive semi-definite since~\eqref{eq:kernel_x_y} shows for arbitrary pairwise distinct points 
$\{x_1,\ldots, x_N\} \subset \mathcal X$, $N\ge 1$ that for all $\alpha \in \R^N\setminus \{0\}$,
\ben\label{eq:pos_semi}
\begin{split}
    \sum_{i,j=1}^N  \alpha_i \alpha_j k_{ij} & = 
    \sum_{i,j=1}^N  \alpha_i \alpha_j (k(x_i, 
    \cdot), k(x_j,\cdot))_{\mathcal H(\mathcal X)} 
    \\
    &= \left( \sum_{i=1}^N \alpha_i k(x_i, 
    \cdot), \sum_{i=1}^N \alpha_i 
    k(x_i,\cdot)\right)_{\mathcal H(\mathcal X)}\ge 0.  
\end{split}
\een
We conclude that reprodcuing kernels are symmetric and positive semi-definite. It also 
follows from \eqref{eq:pos_semi} that 
a reproducing kernel is positive definite if and only if the evaluation maps $\delta_y$ are linearly 
independent in $(\mathcal H(\mathcal X))^*$ for all $y\in \mathcal X$. 
The Moore-Aronszajn theorem ensures that for each symmetric and positive semi-definite kernel $k$ there is a 
unique RKHS. 
\begin{theorem}\label{thm:reproducing_existence}
    Suppose that $k:\mathcal X\times \mathcal X \rightarrow \R$ is a positive semi-definite and symmetric scalar kernel. Then there exists a unique Hilbert
    space $\mathcal H(\mathcal X)$ of real valued functions 
    $f:\mathcal X\rightarrow \R$ for which $k$ is the reproducing kernel. 
\end{theorem}
\begin{proof}
    We refer the reader to \cite{MR0051437} and \cite[p.138, Thm.10.10]{MR2131724}.
\end{proof}
We refer to \cite[Theorem 10.12, p.139]{MR2131724} for a 
more explicit characterisation  of RKHS generated by scalar positive definite 
kernels in the case $\mathcal X=\R^d$.

Similarly to scalar kernels we  define matrix-valued kernels:
\begin{definition}\label{def:kernel_vector}
    
    \begin{itemize} 
        \item[(a)]  A function $K:\mathcal X \times \mathcal X \rightarrow \R^{d,d}$ is call a 
        symmetric and positive (semi)-definite \emph{matrix kernel} if 
        \begin{itemize}
            \item[($a_1$)] $\forall x,y\in \mathcal X$, $K(x,y)=K(y,x)$
            \item[($a_2$)] for arbitrary distinct points 
            $\{x_1,\ldots, x_N\} \subset \mathcal X$, $N\ge 1$, the matrix  $K_{ij}:=K(x_i,x_j)$ 
            satisfies, for all $\alpha_1, \alpha_2,\ldots, 
            \alpha_N\in \R^d$, not all of them identically zero,
            $$	 \sum_{i,j=1}^N K_{ij}\alpha_i\cdot \alpha_ j > (\ge ) 0.  $$
        \end{itemize}
        \item[(b)] A kernel $K$ is called radial scalar kernel if there 
        exists a function $\gamma:\R\rightarrow \R^{d,d}$ such that $K(x,y)= \gamma(|x-y|)$ for all 
        $x,y\in \mathcal X$.
        \item[(c)] A function $K:\mathcal X\times \mathcal X\rightarrow \R^{d,d}$ is called \emph{matrix-valued reproducing kernel} if
        \begin{itemize}
            \item[$(c_1)$] for every  
            $x\in \mathcal X$ and every $a\in \R^d$,
            $
            K(x,\cdot)a\in \Hrep
            $
            \item[$(c_2)$] for all $ f\in \Hrep $  and for all $a\in \R^d$,
            \ben		
            (K(x,\cdot)a, f)_{\Hrep} = (a\otimes \delta_x) f = a\cdot f(x).	
            \een
        \end{itemize} 
        
    \end{itemize}
\end{definition}
Unlike scalar reproducing kernels, matrix-valued reproducing kernels are not necessarily symmetric. 
However, using the reproducing property $(c_2)$ repetively yields
\ben\label{eq:nonsym_K}
\begin{split}
    K(x,y)a\cdot b & =   (K(y,\cdot)b, K(x,\cdot)a)_{ \Hrep}
    = (K(x,\cdot)a, K(y,\cdot)b)_{ \Hrep}
    = K(y,x)b\cdot a,
\end{split}
\een
for all $a,b\in \R^d$ and all $x,y\in \mathcal X$
so that for all $x,y\in \mathcal X$, we get
$ K(x,y) = K(y,x)^\top. $
Hence, assuming that $K$ is symmetric, using \eqref{eq:nonsym_K} we obtain that for arbitrary distinct points 
$\{x_1,\ldots, x_N\} \subset \mathcal X$, $N\ge 1$, the matrix  $K_{ij}:=K(x_i,x_j)$ 
satisfies,
\ben
\begin{split}
    \sum_{i,j=1}^N K_{ij}\alpha_i\cdot \alpha_ j  & = \sum_{i,j=1}^N 
    \alpha_i\cdot \alpha_ j  (K(x_i,\cdot)\alpha_i, 
    K(x_j,\cdot)\alpha_j)_{ \Hrep} \\
    & = \left( \sum_{i=1}^N \alpha_i K(x_i,\cdot)\alpha_i, 
    \sum_{i=1}^N K(x_i,\cdot)\alpha_i \right)_{ \Hrep} \ge 0,
\end{split}
\een
for all $\alpha_1, \alpha_2,\ldots, \alpha_N\in \R^d$, not all of them identically zero.
Consequently, every symmetric reproducing kernel 
$K:\mathcal X\times\mathcal X\rightarrow \R^{d,d}$ is also positive 
semi-definite. 

Similarly to the scalar case it holds
\begin{theorem}
    For every  matrix-valued symmetric 
    and positive semi-definite kernel $K:\mathcal X\times \mathcal X\rightarrow \R^{d,d}$ 
    there exists a unique Hilbert
    space of vector valued functions $\mathcal H(\mathcal X,\R^d)$ for which $K$ 
    is the matrix-valued reproducing kernel. 
\end{theorem}
\begin{proof}
    We refer to Proposition 1 in \cite{MR2265340}.	
\end{proof}
Another special property of vvRKHSs is that for all $a\in \R^d$ and 
$x\in \mathcal X$, the 
evaluation map
$$  \Hrep\rightarrow \R:\; f\mapsto (a\otimes\delta_x) f = f(x)\cdot a  $$
is continuous. In fact, we obtain from $(c_2)$ and Cauchy's inequality that
\begin{align*} 
|a\cdot f(x)| =  |(K(x,\cdot)a, f)_{\Hrep}| \le \|K(x,\cdot)a\|_{\Hrep} \|f\|_{\Hrep}.
\end{align*}
Conversely, for every Hilbert space $\Hrep$ 
of vector 
valued functions $f:\mathcal X\rightarrow \R^d$ for which the evaluation map $f\mapsto f(x)\cdot a$ is 
continuous for all $a\in \R^d$ and $x\in \mathcal X$, there is a unique 
kernel $K(x,y)$ satisfying $(c_1)$ and $(c_2)$; cf. \cite[p.143, Thm. 10.2]{MR2131724}.

\begin{example}\label{ex:Hk_RKHS}
    As an example consider $\mathcal X=\R^d$ and the Sobolev space $H^k(\R^d)$ 
    with $k,d$ being non-negative integers satisfying $k \ge \lfloor\frac{d}{2}\rfloor + 1$. Then, 
    the Sobolev embedding yields
    $$ \forall \varphi \in H^k(\R^d), \quad \|\varphi\|_{C^0(\R^d)} \le c \|\varphi\|_{H^k(\R^d)}. 	$$
    Thus, the point evaluation 
    $\delta_y: H^k(\R^d) \rightarrow \R, f \mapsto f(y)$ is in fact
    continuous 
    for all $y\in \R^d$. Hence there is a reproducing kernel $k:\R^d\times \R^d \rightarrow \R$ for which $H^k(\R^d)$ is the 
    reproducing kernel Hilbert space.  
\end{example}

We depict some examples of positive semi-definite kernels in the following:
\begin{example}
    \begin{align}
    \label{eq:kernel1}	 K_1(x,y) & :=  \si(|x-y|)  && (\textsf{kernel generating Paley-Wiener space}),\\
    \label{eq:kernel2}      K_2(x,y) & := e^{-|x-y|^2}  &&  (\textsf{Gauss kernel}), \\ 
    \label{eq:kernel3}      K_3(x,y) & := e^{-|x-y|}  &&  (\textsf{Laplacian kernel}),\\
    \label{eq:kernel4}      K_4(x,y) & := |x-y|^{k-d/2}B_{k-d/2}(|x-y|) && (\textsf{kernel 
        generating } W^k_2(\R^d)),\\ 
    \label{eq:kernel5}      K_5(x,y) & := (1-|x-y|)_+^4(4|x-y|+1) && 
    (\textsf{polynomial kernel with compact support}).
    \end{align}
  The function $B_{k-d/2}$ is called Hankel function and 
    $\si(x):=\sin(x)/x$. 
\end{example}

One special feature of RKHS/vvRKHS is that the convergence in $\mathcal H(\mathcal X,\R^d)$ implies 
pointwise converges on $\mathcal X$; cf. \cite{MR2131724}. 
Another property is that the span of $K(x,\cdot)a$, $a\in \R^d$, 
$x\in \mathcal X$ is dense in $\Hrep$ in case $\mathcal X$ is open. We recall both 
results in the following lemmas.     
\begin{lemma}
    Let $\mathcal X$ be a compact set and $K$ a matrix-valued symmetric and 
    positive definite kernel on $\mathcal X$ and $\Hrep$ the corresponding vvRKHS. 
    Then the span of $\{K(x,\cdot)a:\; x\in \mathcal X\}$ is dense in 
    $\Hrep.$
\end{lemma}
\begin{proof}
    Let $V$ denote the closure of $\textsf{span}\{K(x,\cdot)a:\; x\in 
    \mathcal X\}$ in $\Hrep$. Since $\Hrep$ is a Hilbert space 
    it holds $\Hrep = V \oplus V^\bot$. Let $f \in V^\bot$ be arbitrary. Then
    the reproducing property yields $(f,K(x,\cdot)e_i)=f_i(x)=0$ for all $x\in \mathcal X$ and 
    $i=1,\ldots, d$. It follows $f=0$ and thus $V^\bot=\emptyset$ and 
    consequently $V=\Hrep$. 
\end{proof}
\begin{lemma}
    Suppose $\Hrep$ is a vvRKHS with matrix-valued kernel 
    $K:\mathcal X\times \mathcal X \rightarrow \R^{d,d}$.  Then, if $f_n,f\in \Hrep$ 
    with $f_n\rightarrow f$ as $n\rightarrow \infty$ in $ \Hrep$, it follows
    $$ f_n(x)\rightarrow f(x) \quad \textsf{ for all } x\in \mathcal X. $$
\end{lemma}
\begin{proof}
    For all $e_i$ with $i\in \{1\ldots, d\}$ it holds
    \begin{align*}
    |(f(x)-f_n(x))\cdot e_i|  = |(f-f_n, K(x,\cdot)e_i)_{ \Hrep}| \le \|f-f_n\|_{\Hrep}\|K(x,\cdot)e_i\|_{\Hrep}.
    \end{align*}
\end{proof}

\subsection{Formulas of shape gradients in reproducing kernel Hilbert spaces}\label{subsec:explicit}
This section presents the central part of this paper. We give explicit formulas for 
shape gradients in reproducing kernel Hilbert spaces and study special 
radial kernels. 
Moreover, we discuss methods to approximate and discretise the domain expression of the
shape derivative on various finite dimensional reproducing kernel Hilbert spaces constructed by finite 
elements and kernels.
It turns out that the
gradient of the shape derivative in a vvRKHS can be recovered by a 
sequence of vector solved on these finite dimensional subproblems. 
In a number of recent articles \cite{MR2642680,MR3348199,MR3436555,lauraindistributed}, the
volume expression
has been used successfully by employing finite elements. 
Subsequently, we set this finite element method in a broader context and 
relate it to reproducing kernel Hilbert spaces.

In this section we consider shape 
differentiable functions
\ben\label{eq:shape_func}
J:\Xi\subset \wp(D) \rightarrow \R, \quad \Omega\mapsto J(\Omega) 
\een
for open and bounded $D\subset \R^d$ that admit for each $\Omega$ in $\Xi$ a tensor representation of the form 
\ben\label{eq:tensor_kernel}
dJ(\Omega)(X) = \int_{\mathcal X} \Sb_1 : \partial X +  \Sb_0 \cdot X\, dx,   
\een
where $\Sb_1\in L_1(\mathcal X, \R^{d,d})$ and $\Sb_0\in L_1(\mathcal X, \R^d)$. 
This means the shape 
derivative is a linear and continuous mapping 
$dJ(\Omega):W^1_\infty(\mathcal X,\R^d)\rightarrow \R$. 
Typically, the set
$\mathcal X$ is either $\Omega$ or $D$; cf. Example~\ref{ex:shape_func_f}. 
\subsubsection*{Shape gradients in vvRKHS} 
Reproducing kernel 
Hilbert spaces allow us to obtain explicit formulas for the Riesz 
representation of functionals defined on them as shown by the following lemma.  
\begin{lemma}\label{lem:first_order_kernel}
    Let $\mathcal X\subset \overbar{D}$.	Suppose $\Hrep $ is a vvRKHS with matrix-valued kernel 
    $K(x,y)=(K_1(x,y),\ldots,K_d(x,y))$ and assume $dJ(\Omega)\in (\Hrep)^*$.  Then the gradient $\nabla 
    J(\Omega)$ of $J$ at $\Omega$ with respect to the $\Hrep
    $-metric is given pointwise for all $y\in \mathcal X$ by
    \ben\label{eq:formula_gradient_H}	
    \nabla J(\Omega)(y) = \sum_{i=1}^d\left(\int_{\mathcal X} \Sb_1(x) : \partial_x K_i(x,y)   + \Sb_0(x) \cdot  K_i(x,y)  \, dx \right)e_i, 
    \een
    where $e_i$ denotes the standard basis of $\R^d$. 
    %Using the deviator $\Sb_1^D$,
    %we obtain an equivalent formula
    %\ben
    %\nabla J(\Omega)(y) = \sum_{i=1}^d\left(\int_{\mathcal X} \Sb_1^D(x) : \partial_x 
    %		K_i(x,y)  + \frac{1}{d}\tr(\Sb_1(x)) \divv_x(K_i(x,y)) + \Sb_0(x) \cdot  K_i(x,y)  \, dx \right)e_i.
    %  \een
\end{lemma}
\begin{proof}
    Let $e_i\in \R^d$, $i\in \{1,2,\ldots, d\}$, denote the standard basis 
    of $\R^d$. By definition the gradient $\nabla J(\Omega)$ in 
    $\Hrep$ satisfies
    \ben\label{eq:gradient_H}
    (\nabla J(\Omega), \varphi )_{\Hrep} = dJ(\Omega)(\varphi) \quad \textsf{ for all }\varphi\in \Hrep. 	
    \een
    By property $(c_1)$ we know that for all $y\in 
    \mathcal X$ the function $\varphi_i^y(\cdot):=K(y, \cdot)e_i$ belongs 
    to  $\Hrep$. Therefore, plugging $\varphi_i^y$ into  
    \eqref{eq:gradient_H} and using the reproducing property $(c_2)$, we obtain
    \begin{align*}
    dJ(\Omega)(\varphi_i^y) & =  (\nabla J(\Omega), K(y,\cdot)e_i )_{\mathcal H} =   \nabla J(\Omega)(y) \cdot e_i 
    \end{align*} 
    for $i=1,\ldots, d$. 
    This shows 
    $$( \nabla J(\Omega)(y))_i = \int_{\mathcal X} \Sb_1(x) : \partial_x K_i(x,y)   + \Sb_0(x) \cdot  K_i(x,y)  \, dx	$$
    for $i=1,\ldots, d$ and thus completes the proof. 
\end{proof}
\begin{remark}
    \begin{itemize}
        \item Equation \eqref{eq:formula_gradient_H} gives an explicit 
        formula of the gradient $\nabla J(\Omega)$ without any 
        approximation. This is in contrast to the usual method using
        the $H^1$ metric and finite elements; cf. Section 
        \ref{subsec:shape_gradient} and also \cite{lauraindistributed,hiptpagan15}.
        \item For an efficient evaluation of the right-hand side of \eqref{eq:formula_gradient_H}
        we need to approximate the integral over $D$ in an efficient 
        way. Of course, in practice,  \eqref{eq:formula_gradient_H} has usually only to 
        be evaluated for $y$ on the boundary $\partial \Omega$ and not on 
        the whole domain.
        \item The assumption $dJ(\Omega)\in (\Hrep)^*$ is for instance 
        satisfied when $\mathcal X$ is open and $ \Hrep$ is continuously embedded into
        $C^1(\overbar{\mathcal X},\R^d)$, i.e., there is a constant $c>0$, such that
        $$  \|\varphi\|_{C^1(\overbar{\mathcal X},\R^d)} \le c \|\varphi\|_{\Hrep} \quad \textsf{ for all }\varphi \in  \Hrep. $$ 
        Similarly as in Example~\ref{ex:Hk_RKHS}, 
        in case $\mathcal X$ open, bounded and of class $C^1
        $,  one could consider the Sobolev space $H^k(\mathcal X)$ 
        with integers $k,d\ge 1$ satisfying $k \ge \lfloor\frac{d}{2}\rfloor + 2$. Then 
        the Sobolev embedding shows
        $$ \forall \varphi \in H^k(\mathcal X), \quad \|\varphi\|_{C^1(\mathcal X)} \le c \|\varphi\|_{H^k(\mathcal X)} 	$$
        and consequently 
        $dJ(\Omega)\in (\mathcal H(\mathcal X,\R^d))^*$, where $\mathcal H(\mathcal X,\R^d) := [H^k(\mathcal X)]^d$. 
    \end{itemize}
\end{remark}

\subsubsection*{Radial kernels}
We now focus on radial kernels of the form
\ben\label{eq:kernel_radial_cor}
K(x,y) = \phi_\sigma(|x-y|^2)I, \quad \phi_\sigma(r) := \phi(r/\sigma), \quad\sigma>0,	
\een
where $\phi\in C^1(\R^d)$ is some given function. 

\begin{lemma}\label{lem:product}
    Assume $k$ is a reproducing kernel on the set 
    $\mathcal X\subset \overbar{D}$ with corresponding reproducing kernel Hilbert space $\mathcal H(\mathcal X)$. Then $K(x,y):=k(x,y)I$ 
    is a matrix-valued reproducing kernel with vector valued reproducing
    kernel Hilbert space 
    $\mathcal H(\mathcal X;\R^d) := [\mathcal H(\mathcal X)]^d $ and inner product
    \ben\label{eq:inner_radial}
    (f,g)_{\mathcal H(\mathcal X;\R^d)} := (f_1,g_1)_{\mathcal H(\mathcal X)} + \cdots + (f_d,g_d)_{\mathcal H(\mathcal X)}	
    \een
    for all $f=(f_1,\ldots, f_d)$ and $g=(g_1,\ldots, g_d)$ with $f_1,\ldots, f_d, g_1,\ldots, g_d\in \mathcal H(\mathcal X)$. 
\end{lemma}
\begin{proof}
    We have to show that $K(x,y)$ is the reproducing kernel for the Hilbert space $[\mathcal H(\mathcal X)]^d$ with inner producing given by \eqref{eq:inner_radial}. Clearly $K(x,y)$ satisfies $(c_1)$, so it remains to show $(c_2)$.
    By assumption, $k$ is a scalar reproducing 
    kernel satisfying
    \ben\label{eq:kernel_prop}
    \forall x\in \mathcal X, \; \forall f\in \mathcal H(\mathcal X), \quad (k(x,\cdot)a,f)_{\mathcal H(\mathcal X)} = f(x).
    \een
    Then for all $f=(f_1,\ldots, f_d)$, $f_1,\ldots, f_d\in \mathcal H(\mathcal X)$ 
    and for all $a = (a_1,\ldots, a_d)\in \R^d$, we get
    \ben
    \begin{split}
        (K(x,\cdot)a,f)_{\mathcal H(\mathcal X;\R^d)} & = (a_1 k(x,\cdot),f)_{\mathcal 
            H(\mathcal X)} + \cdots + (a_d k(x,\cdot),f_d)_{\mathcal H(\mathcal X)} \\
        &\stackrel{\eqref{eq:kernel_prop}}{=} a_1 f_1(x) + \cdots + a_d f_d(x)\\
        & = a\cdot f(x) = a\otimes \delta_x f
    \end{split}
    \een
    and this shows $(c_2)$.
\end{proof}
\begin{example}
    Let us return to Example~\ref{ex:Hk_RKHS} where $\mathcal X=\R^d$ and $
    \mathcal H(\mathcal X) = H^k(\R^d)$ with 
    $k \ge \lfloor\frac{d}{2}\rfloor + 1$. Let $k(x,y)$ be the scalar 
    reproducing kernel 
    associated with  $H^k(\R^d)$. Then according to Lemma~\ref{lem:product} 
    the matrix-valued radial kernel $K(x,y):= k(x,y)I$ is the reproducing kernel
    for $\mathcal H(\R^d,\R^d):=[H^k(\R^d)]^d$. 
\end{example}

\begin{lemma}
    Let $\phi\in C^1(\R)$ be such that $k(x,y):=\phi_\sigma(|x-y|^2)$, $\sigma>0$, is 
    a reproducing kernel on $\mathcal X\subset \overbar{D}$ with reproducing 
    kernel Hilbert space $\mathcal H(\mathcal X)$. Let 
    $[\mathcal H(\mathcal X)]^d$ be the vector valued kernel Hilbert space 
    for  radial kernel $K(x,y)$ given by 
    \eqref{eq:kernel_radial_cor} and assume $dJ(\Omega)\in ([\mathcal{H}(\mathcal{X})]^d)^*$. Then the gradient $\nabla^\sigma J(\Omega)
    $ in $[\mathcal H(\mathcal X)]^d$ is given pointwise in $\mathcal X$ by 
    \ben\label{eq:formula_radial_kernel}
    \nabla^\sigma J(\Omega)(y) =\int_{\mathcal X} \left(\phi_\sigma(|x-y|^2) 
    \Sb_0(x)+\frac{2}{\sigma}\phi_\sigma'(|x-y|^2) \Sb_1(x) (x-y)  \right)\ dx, 
    \een
    where $\phi_\sigma'(r) := \phi'(r/\sigma)$. 
\end{lemma}
\begin{proof}
    This follows directly from Lemma~\ref{lem:first_order_kernel}.
\end{proof}

\begin{corollary}\label{cor:divergence_radial}
    Let $\phi\in C^2(\R)$ be as in the previous lemma and suppose that 
    $\mathcal X\subset \overbar{D}$ is open.	
    The gradient of $ \nabla^\sigma J(\Omega)$ is given pointwise in $\mathcal X$ by 
    \ben\label{eq:formula_grad}
    \begin{split}
        \partial_y(\nabla^\sigma J(\Omega))(y)  =& - \int_{\mathcal X} \frac{2}{\sigma} \left( 
        \phi_\sigma'(|x-y|^2)( \Sb_0(x)\otimes (x-y) + \Sb_1(x)\right) \, dx\\
        & - \int_{\mathcal X} \frac{4}{\sigma^2} \phi_\sigma''(|x-y|^2) 
        (\Sb_1(x)(x-y))\otimes (x-y)\, dx 
    \end{split}
    \een
    and thus there is a constant $c>0$, so that for all $\sigma >0$, 
    \ben\label{eq:grad_estimate}
    \|\partial_y(\nabla^\sigma J(\Omega))\|_{L_\infty(\mathcal X,\R^{d,d})} \le \frac{c}{\sigma}. 
    \een
\end{corollary}
\begin{proof}
    Equation \eqref{eq:formula_grad} follows by direct computation from 
    \eqref{eq:formula_radial_kernel}.
    
    We now prove \eqref{eq:grad_estimate}. Since $\phi$ is $C^2$, it (and its first and second derivative) attains its maximum on the closed 
    unit ball $\overbar B_1(0)$ centered at the origin in $\R^d$. 
    Let $r:=\textsf{diam}(\overbar{D})$ denote the (finite) diameter of $\overbar{D}$. 
    Then for all $\sigma \ge r$ and all $x,y\in \overbar{D}$ we have
    $ \left|\frac{x-y}{\sigma}\right|\le 1.  $
    Hence, there is a constant $C>0$ so that for all $\sigma\ge r$,
    $$
    \sup_{x,y\in \overbar{D}} |\phi_\sigma(x-y)| + \sup_{x,y\in \overbar{D}} |\phi_\sigma'(x-y)|  + \sup_{x,y\in \overbar{D}}|\phi_\sigma''(x-y)| \le C.
    $$
    Thus, we obtain ($\Sb_1$ and $\Sb_0$ are extended by zero outside of 
    $\mathcal X$)
    \begin{align*}  
    |\divv_y\nabla^\sigma J(\Omega)(y)|  \le &  \int_D \frac{2}{\sigma} \left(|\phi_\sigma'(|x-y|^2)|( |\Sb_0(x)||x-y| +  |\Sb_1(x)|\right) \, dx\\
    & + \int_D \frac{4}{\sigma^2} \phi_\sigma''(|x-y|^2) 
    |\Sb_1(x)||x-y|^2\, dx \\
    \le & \int_D \frac{2}{\sigma} \left( |\Sb_0(x)||x-y| +  
    |\tr(\Sb_1(x))|\right) + \frac{4}{\sigma^2} |\Sb_1(x)||x-y|^2\, dx\\
    \le & \frac{C}{\sigma} \int_D  \left(|\Sb_0(x)| +  
    |\tr(\Sb_1(x))|\right) + \frac{C}{\sigma^2}\int_D  |\Sb_1(x)|\, dx,
    \end{align*}
    where the constant $C$ only depends on $r$. Finally taking the supremum on 
    both sides and passing to the limit $\sigma \searrow 0$ gives the 
    desired result \eqref{eq:grad_estimate}. 
\end{proof}

\begin{corollary}
    Let $\phi\in C^2(\R)$ be as in the previous lemma and suppose that 
    $\mathcal X\subset \overbar{D}$ is open.	
    Then the divergence of $\nabla^\sigma J(\Omega)
    $ is given pointwise in $\mathcal X$ by 
    \ben\label{eq:formula_div}
    \begin{split}
        \divv_y(\nabla^\sigma J(\Omega))(y)  =& - \int_{\mathcal X} \frac{2}{\sigma} \left( 
        \phi_\sigma'(|x-y|^2)( \Sb_0(x)\cdot (x-y) +  \tr(\Sb_1(x))\right) \, dx\\
        & - \int_{\mathcal X} \frac{4}{\sigma^2} \phi_\sigma''(|x-y|^2) 
        \Sb_1(x)(x-y)\cdot (x-y)\, dx.  
    \end{split}
    \een
    Moreover, there is a constant $c>0$, so that for all $\sigma >0$, 
    \ben\label{eq:divv_estimate}
    \|\divv_y\nabla^\sigma J(\Omega)\|_{L_\infty(\mathcal X)} \le \frac{c}{\sigma}  
    \een
    and
    \ben\label{eq:conv_div}
    \divv_y\nabla^\sigma J(\Omega) \rightarrow 0 \quad \textsf{ in } L_\infty(\mathcal X,\R^d) \textsf{ as } \sigma \nearrow \infty.  
    \een
\end{corollary}
\begin{proof}
    Using the tensor relations $A:b\otimes c= b\cdot Ac$ and 
    $(a\otimes b)c = a (b\cdot c)$ for all $A\in \R^{d,d}$ and 
    $a,b,c\in \R^d$ and $\divv(v)= \partial v :I$, we 
    infer formula \eqref{eq:formula_div} directly from 
    \eqref{eq:formula_grad}. The rest of the proof is obvious.  
\end{proof}

\begin{remark}
    Formula \eqref{eq:formula_radial_kernel} in conjunction with Corollary~\ref{cor:divergence_radial} allows us to 
    interprete the terms $\Sb_0$ and $\Sb_1$. The term $\Sb_0$ is responsible for "translations" while
    $\Sb_1$ allows for shape deformations.
\end{remark}

We now consider the Gauss kernel for which \eqref{eq:formula_radial_kernel} and \eqref{eq:formula_div} further 
simplify. 
\begin{corollary}
    For the Gauss kernel $K(x,y) := e^{-(x-y)^2/\sigma}I$, $\sigma>0$, the gradient of 
    $J(\Omega)$ at $\Omega$ is given pointwise by 
    \ben\label{eq:gradient_gauss}
    \nabla^\sigma J(\Omega)(y) =\int_D e^{-|x-y|^2/\sigma} \left( 
    \Sb_0(x)-\frac{2}{\sigma} \cdot \Sb_1(x)(x-y) \right)\, dx.
    \een
    Moreover, the divergence is given by 
    $$ \divv_y(\nabla^\sigma J(\Omega)) = \int_D \frac{2}{\sigma} e^{-|x-y|^2/\sigma} \left( 
    \Sb_0(x)\cdot (x-y) - \frac{2}{\sigma} \Sb_1(x)(x-y)\cdot (x-y) 
    +  \tr(\Sb_1(x))\right)\, dx .
    $$
\end{corollary}

\subsection{Finite dimensional reproducing kernel Hilbert spaces} \label{subsec:shape_gradient}
In this section, $J$ is a shape function defined on a subset $\Xi$ of 
$\wp(D)$, $D\subset \R^2$, i.e., we now focus on the two 
dimensional case $d=2$.
Recall our generic assumption that in an open subset $\Omega$ of $D$, the shape derivative is given by \eqref{eq:tensor_kernel}. 
Subsequently, we want to discuss the relation between a finite element space 
and RKHS and spaces generated by radial kernel functions. In the previous 
section, we always started with a reproducing kernel.  Here, we assume
that a finite dimensional Hilbert space is given and we seek
the reproducing kernel. 

\subsubsection*{Reproducing kernels associated with a finite dimensional space $\mathcal V_N(\mathcal X,\R^2)$} 
For a given set  $\mathcal X\subset \overbar{D}$,  let  
$\mathcal V_N(\mathcal X,\R^2)$ be some finite dimensional space of vector 
valued functions defined on $\mathcal X$ and contained in 
$C(\overbar D,\R^d)\cap W^1_1(D,\R^d)$. We assume $$ \{v^1, v^2, \ldots, v^{2N}\}  \textsf{ is a basis (not necessarily orthonormal) of } \mathcal V_N(\mathcal X,\R^2).$$ 
Suppose an inner product $(\cdot,\cdot )_{\mathcal V_N(\mathcal X,\R^2)}$ on $\mathcal V_N(\mathcal X,\R^2)$.
Then $\mathcal V_N(\mathcal X,\R^2)$ is a reproducing kernel Hilbert space 
with the $i$th raw $K_i(x,y) = K(\cdot,y)e_i$ (for $y$ fixed) of the kernel $K(x,y)$ defined as the solution of
$$ (K_i(\cdot,y), X)_{\mathcal V_N(\mathcal X,\R^2)} = X(y)\cdot e_i, \quad \textsf{ for all } X \in \mathcal V_N(\mathcal X,\R^2).	$$
Since $K(x,y)=(K(x,y))^\top$, it follows 
$K^\top(y,\cdot)e_i\in \mathcal V_N(\mathcal X,\R^2)$. 
Then the gradient $\nabla J(\Omega) \in \mathcal V_N(\mathcal X,\R^2)$ of $J$ at 
$\Omega$ is given by 
$$   dJ(\Omega)(X) = (\nabla J(\Omega), X )_{\mathcal V_N(\mathcal X,\R^2)} \quad \textsf{ for all } X \in \mathcal V_N(\mathcal X,\R^2).  $$
For the numerical realisation it is beneficial to have an explicit formula for the gradient in terms of the basis elements: let $A_N := (v^j,v^i)_{\mathcal V_N(\mathcal X,\R^2)}$, $F_N := (dJ(\Omega)(v^1), \ldots, dJ(\Omega)(v^{2N}))^\top$ and $\alpha := (\alpha_1, \ldots, \alpha_{2N})^\top $, then 
\ben\label{eq:formula_gradient_finite}
\nabla J(\Omega) = \sum_{k=1}^{2N} \alpha_k v^k, \qquad \alpha = A_N^{-1}F_N . 
\een
Of course this formula gives the same gradient as \eqref{eq:formula_gradient_H}, i.e., 
\ben
\nabla J(\Omega)(y) = \sum_{i=1}^2\left(\int_{\mathcal X} \Sb_1(x) : \partial_x (K(x,y)e_i)   + \Sb_0(x) \cdot  (K(x,y)e_i)  \, dx \right)e_i, \qquad \textsf{for }i=1,2.  
\een

\subsubsection*{Metrics on $\mathcal V_N(\mathcal X,\R^2)$}
Usually the space $\mathcal V_N(\mathcal X,\R^2)$ is contained in some 
Hilbert space $ \mathcal H(\mathcal X,\R^d) $. Therefore it is natural to 
equip the space  $\mathcal V_N(\mathcal X,\R^2)$ with the inner product 
$(\cdot,\cdot)_{\Hrep}$ from the space $\Hrep$ and to compute the gradient $\nabla J(\Omega)$ with respect to this inner product,
\ben
(\nabla J(\Omega), \varphi)_{\mathcal H(\mathcal X,\R^2)} = \int_D \Sb_1 : \partial \varphi + \Sb_0\cdot \varphi \; dx \quad \textsf{ for all }\varphi\in \mathcal V_N(\mathcal X,\R^2). 	
\een
\begin{example}
    For instance in case $\mathcal X=D\subset \R^2$, the space $\mathcal V_N(\mathcal X,\R^2)$ could comprise 
    conforming Lagrange finite elements contained in $\ac{H}^1(D, \R^2)$; see also below. 
    Then
    \ben
    (\nabla J(\Omega), \varphi)_{H^1(D,\R^2)} = \int_D \Sb_1 : \partial \varphi + \Sb_0\cdot \varphi \; dx \quad \textsf{ for all }\varphi\in \mathcal V_N(D,\R^2). 	
    \een
    In case $dJ(\Omega)$ is supported on $\partial\Omega$, i.e., if $\Sb_0=0$ and $\Sb_1=0$ on $D\setminus \overline \Omega$, then also $\mathcal X=\Omega$ 
    would be an admissible choice. In this case it is sufficient to solve the above 
    variational problem on the domain $\Omega$. 
\end{example}

Given a space $\mathcal V_N(\mathcal X,\R^2)$ as above, the simplest metric on 
it (not induced by the ambient space) can be defined on the basis elements $v^i$ by 
\ben
(v^i,v^j)_{\mathcal V_N(\mathcal X,\R^2)} := \delta_{ij}, \quad i,j\in \{1,2,\ldots, 2N\}.	
\een
More generally, for arbitrary $v,w\in \mathcal V_N(\mathcal X,\R^2)$ we find by definition
$\alpha_i,\beta_i$ in $\R$, $i,j=1,2,\ldots, 2N$,
\ben
v = \sum_{i=1}^{2N} \alpha_i v^i, \quad w = \sum_{i=1}^{2N} \beta_i v^i. 
\een
Then we set
\ben\label{eq:euclid_metric}
\left(v, w \right)_{\mathcal V_N} := \sum_{i,j=1}^{2N}\alpha_i\beta_j \delta_{ij}. 
\een
We will refer to this metric as Euclidean metric.
The gradient of $J$ with respect to this metric is given by 
\ben\label{eq:gradient_eulid}
\nabla J(\Omega) = \sum_{k=1}^{2N} dJ(\Omega)(v^k) v^k. 
\een
It can be readily checked that with the Eulidean metric, the reproducing kernel has 
the form
\ben
K(x,y) = \sum_{l=1}^{2N} v^l(y)\otimes  v^l(x).	 
\een
So $K(x,y)$ is not radial kernel, but it is has only non-zero entries on the diagonal.
An interesting application of the choice ``$\mathcal V_N(D,\R^2)$ = linear Lagrange 
finite elements on $D$'' equipped with the Eulidean metric was proposed in \cite{MR800331}; see also the next section. 
The authors use as descent direction the negative of the gradient defined in 
\eqref{eq:gradient_eulid}  to obtain optimal 
triangulations involving second 
order elliptic PDEs.

\subsubsection*{Building space $\mathcal V_N(\mathcal X,\R^d)$ with finite element spaces}\label{subsect:fem}
Maybe the easiest way to construct a basis for $\mathcal V_N(\mathcal X,\R^2)$ is to use finite elements. 
For simplicity we assume that $\mathcal X$ is a polygonal set. 
Let  $\{\mathcal T_h\}_{h>0}$ denote a family of simplicial triangulations $\mathcal T_h=\{K\}$ consisting 
of triangles $K$ such that 
\ben
\overbar{\mathcal X} = \bigcup \limits_{K\in \mathcal T_h}  K, \quad \forall h >0.
\een
For every element $K\in \mathcal T_h$, $h(K)$ denotes the diameter of  $K$ 
and $\rho(K)$ is the diameter of the largest ball contained in $K$. The maximal 
diameter of all elements is denoted by $h$, i.e., 
$h:=\textnormal{max} \{h(K) \  | \ K\in \mathcal T_h\}.$ Each 
$K\in \mathcal T_h$ consists of three nodes and three edges and we denote the 
set of nodes and edges by 
$\mathcal N_h$ and $\mathcal E_h$, respectively. We assume that 
there exists a positive constant $\varrho >0$, independent of $h$,   such 
that 
\ben
\frac{h(K)}{\rho(K)} \le \varrho
\een
holds for all elements $K \in \mathcal T_h$ and all $h>0$.   Then we 
may define Lagrange finite element functions of order $k\ge 1$ by 
\ben\label{eq:V_h_k}
V_h^k(\mathcal X) := \{v\in C(\overbar{\mathcal X}):\; v_{|K}\in \mathcal P^k(K), \;\forall  K\in \mathcal T_h \}.	
\een
Recall that in the linear case $k=1$ a basis $v^i\in C(\overbar{\mathcal X})$ may be defined via
\ben
x_j\in \mathcal N_h, \quad  v^i(x_j) = \delta_{ij}, \quad i,j=1,\ldots, N_h, 	
\een
where $\delta_{ij}$ denotes the Kroenecker symbol.
We can then define 
$\mathcal V_{N_h}(\mathcal X,\R^2):=V^k_h(\mathcal X)\times V^k_h(\mathcal X)$, that is,
\ben\label{eq:V_N_h}
\mathcal V_{N_h}(\mathcal X,\R^2)  := \textsf{span}\{ v_1e_1, \ldots, v_{N_h}e_1,  v_1e_2, \ldots, v_{N_h}e_2 \}.
\een

\subsubsection*{Building space $\mathcal V_N(D,\R^d)$ using matrix valued kernels}
Let $\{z_1,z_2,\ldots, z_N\}$ be given points in 
$\mathcal X\subset \overbar{D}$ 
and let $K(x,y)$ be a positive definite and symmetric matrix-valued kernel on 
$\mathcal X$. By 
Theorem~\ref{thm:reproducing_existence}, there exists a Hilbert space 
$\mathcal H(\mathcal X,\R^2)$ for which $K$ is the reproducing kernel.
We define the functions
\ben\label{eq:basis_gauss}
v^i(z) :=   K(z,z_i) e_1\qquad\textsf{and}\qquad  v^{N+i}(z) := K(z,z_i) e_2,  
\een
where $i=1,\ldots, N$ and $\{e_1,e_2\}$ denotes the standard basis of $\R^2$. 
As prototype kernel we take the scaled (radial) Gaussian kernel 
(see \eqref{eq:kernel1}-\eqref{eq:kernel4} for other choices)
\ben
K(x,y) :=  e^{-|x-y|^2/\sigma}I.    
\een
Notice that $K$ is positive definite as shown in \cite[Thm. 6.10,p. 74]{MR2131724}.
By construction the functions $v^i$ decay 
exponentially away from $z^i$. The decay rate is determined by the smoothing 
parameter $\sigma>0$. 
We define the finite dimensional space 
\ben\label{eq:V_N_kernel}
\mathcal V_N(\mathcal X,\R^2) := \textsf{span}\{v^1,v^2, \ldots, v^{2N-1}, v^{2N} \}. 
\een
In case $\mathcal X$ is open,
$\mathcal V_N(\mathcal X,\R^2)\subset C^\infty(\mathcal X,\R^2)$.

Recall that the Gauss kernel $k(x,y) = e^{-|x-y|^2/\sigma}$ is a positive 
reproducing kernel (which can be seen by using Fourier transform). Hence, according to Lemma~\ref{lem:product} 
$K(x,y) = e^{-|x-y|^2/\sigma}I$ is a matrix-valued symmetric and positive 
definite reproducing 
kernel. 
The elements of 
$\mathcal V_N(\mathcal X,\R^d)$ are linearly independent and 
also $v^i\in \mathcal H(\mathcal X,\R^d)$.

\subsubsection*{Limiting case $N\rightarrow \infty$ for kernel spaces}
Let $\mathcal V_N(\mathcal X,\R^2)$  be the finite dimensional space defined 
in \eqref{eq:V_N_h} and $\mathcal H(\mathcal X, \R^d)$ the vvRKHS of $K(x,y)$.
We are now interested in the behaviour of $\nabla^{\mathcal V_N}J(\Omega)$ as $N$ tends 
to infinity. Denote by $\nabla^{\mathcal H}J(\Omega)$ the solution of
\ben
(\nabla^{\mathcal H} J(\Omega), \varphi)_{\mathcal H} = dJ(\Omega)(\varphi) \quad \textsf{ 
    for all }\varphi \in \Hrep,
\een
and by $\nabla^{\mathcal V_N}J(\Omega)$ the solution of
\ben
(\nabla^{\mathcal V_N} J(\Omega), \varphi)_{\mathcal H} = dJ(\Omega)(\varphi) 
\quad \textsf{ for all }\varphi \in \mathcal V_N(\mathcal X,\R^2). 
\een
We have seen that $\nabla^{\mathcal H} J(\Omega)$ is given 
by the explicit formula \eqref{eq:formula_gradient_H} while 
$\nabla^{\mathcal V_N}J(\Omega)$ can be computed by \eqref{eq:formula_gradient_finite}. 
\begin{lemma}\label{lem:convergence_VN_V}
    There holds
    \ben
    \lim_{N\rightarrow \infty} \|\nabla^{\mathcal V_N} J(\Omega)-\nabla^{\mathcal H} J(\Omega)\|_{\mathcal H(\mathcal X,\R^2)}=0 
    \een
    and thus in particular
    \ben
    \nabla^{\mathcal V_N} J(\Omega)(x) \rightarrow \nabla^{\mathcal H} J(\Omega)(x)\quad \textsf{ for all } x\in \mathcal X \quad \textsf{ as } N\rightarrow \infty. 	
    \een
\end{lemma}
\begin{proof}
    By definition of $\nabla^{\mathcal V_N}J(\Omega)$,
    \ben\label{eq:nabla_H_J}
    (\nabla^{\mathcal V_N}J(\Omega), \varphi)_{\mathcal H} = 
    dJ(\Omega)(\varphi)\quad \textsf{ for all } \varphi\in \VN.  
    \een
    On the one hand, the function $\nabla^{\mathcal V_N}J(\Omega)$ is given by \eqref{eq:formula_gradient_finite}.
    Since by construction 
    $\nabla^{\mathcal V_N} J(\Omega)\in \Hrep$, we may use 
    it as   a test function in \eqref{eq:nabla_H_J}, i.e., 
    \ben  \|\nabla^{\mathcal V_N}J(\Omega)\|_{\Hrep}^2 = dJ(\Omega)(\nabla^{\mathcal V_N}J(\Omega)) \le c \|\nabla^{\mathcal V_N}J(\Omega)\|_{\Hrep},    
    \een
    so that $\|\nabla^{\mathcal V_N}J(\Omega)\|_{\Hrep}\le c$ for all $N$. Hence there is a subsequence $N_k$ and $z\in \Hrep$ such that
    $\nabla^{\mathcal V_{N_k}}J(\Omega) \rightharpoonup z $ weakly in  $\Hrep$. 
    This allows us to pass to the limit in \eqref{eq:nabla_H_J} and we obtain by uniqueness of $\nabla^{\mathcal H}J(\Omega)$,
    \begin{align*} 
    \nabla^{\mathcal V_{N_k}}J(\Omega) &\rightharpoonup \nabla^{\mathcal 
        H}J(\Omega) \qquad \textsf{ weakly in } \Hrep \quad \textsf{ as } k\rightarrow \infty.
    \end{align*}
    Since for every sequence $N\rightarrow \infty$ there is a subsequence $N_k$ such that $\nabla^{\mathcal V_{N_k}}J(\Omega) \rightharpoonup z $ weakly in  $\Hrep$, the whole sequence converges weakly.
    On the other hand, it follows from \eqref{eq:nabla_H_J} that
    \ben
    \|\nabla^{\mathcal V_N}J(\Omega)\|_{\Hrep}^2 = dJ(\Omega)(\nabla^{\mathcal V_N}J(\Omega)) \quad \longrightarrow\quad dJ(\Omega)(\nabla^{\mathcal H}J(\Omega)) = \|\nabla^{\mathcal H}J(\Omega)\|_{\Hrep}^2.
    \een
    Now since weak convergence and norm convergence together imply the strong 
    convergence, the claim follows.
\end{proof}
\begin{remark}	
    Let $\mathcal V_N(D,\R^d)$ be the Lagrange finite element space defined in 
    \eqref{eq:V_N_h} and suppose $\mathcal V_N(D,\R^2)\subset \ac{H}^1(D,\R^2)$.
    It is clear from standard finite element analysis that under suitable smoothness assumptions
    \ben
    \lim_{h\searrow 0}\|\nabla^{H^1} J(\Omega) - \nabla^{\mathcal V_{N_h}}J(\Omega) \|_{H^1(D,\R^2)}=0. 
    \een
    For a proof of this claim for a specific problem we refer to \cite{hiptpagan15}. 
\end{remark}

\section{A linear transmission problem}\label{sec:model_problems}
In this section we discuss a simple cost function constrained by a transmission 
problem. Transmission problems are important for applications because
they can be used to formulate inverse problems such as electrical impedance problems; see \cite{MR2329288,lauraindistributed}. 

\subsection{Problem formulation} 
We are interested in minimising the cost function
\begin{equation}
\label{eq:cost_func}
\min_{\Omega} J(\Omega) = \int_D |u-u_d|^2\; dx \quad \textsf{ over }  \Xi,
\end{equation}
where $\Xi\subset \wp(D)$ is some admissible set and $u=u(\Omega)$ is the 
(weak) solution of the transmission problem
\begin{equation}\label{eq:state_strong_lin}
\begin{split}
-\divv(\beta_+\nabla u^+)  &= f  \quad \textsf{ in } \Omega^+, \\
-\divv(\beta_-\nabla u^-)  &= f \quad \textsf{ in }  \Omega^-, \\
u&=0\quad \textsf{ on } \partial D, 
\end{split} 
\end{equation}
supplemented by the transmission conditions
\begin{equation}\label{eq:state_strong_lin-1}
\beta_+ \partial_n u^+=\beta_- \partial_nu^- \quad \textsf{ and } \quad u^+=u^-\quad \textsf{ on } \Gamma.
\end{equation}
The appearing data in the previous equation is specified by the following assumption. 
\begin{assumption}
    \begin{itemize}
        \item the set $D\subset \R^d$  is a bounded domain with boundary 
        $\partial D$
        \item for every open open subset $\Omega \subset D$, we use the 
        notation $\Omega^+:=\Omega$ and   
        $\Omega^-:=D\setminus \overbar{\Omega}$ 
        \item the \emph{interface} is defined by $\Gamma=\partial{\Omega}^-\cap \partial{\Omega}^+$, so 
        if $\Omega\subset\subset D$, then $\Gamma=\partial\Omega$
        \item the functions $f,u_d$ belong to $H^1(D)$
        \item $\beta^+,\beta^->0$ are positive numbers 
    \end{itemize}
\end{assumption}

Finally let us recall the variational formulation of 
\eqref{eq:state_strong_lin}-\eqref{eq:state_strong_lin-1}
\ben\label{eq:trans_state_lin}
\int_{D} \beta_\chi\nabla u\cdot \nabla v \, dx  =  \int_{D} f  \varphi  \, dx  \quad \textsf{ for all } v \in \ac{H}^1(D),
\een
where $\beta_\chi:=\beta_+\chi + \beta_-(1-\chi)$.
\begin{remark}
    The well-posedness of the optimisation problem \eqref{eq:cost_func} subject to 
    \eqref{eq:trans_state_lin} can be 
    achieved by adding a perimeter term or Sobolev perimeter.  We will not discuss 
    that issue any further here and refer to 
    \cite{MR2731611} and also \cite{SturmHoemHint13,MR3374631}. Other 
    methods to obtain well-posedness include to impose a volume constraint; 
    cf. \cite[p. 225, Section 3.5]{MR3374631}.
\end{remark}
\subsection{Shape derivative}
Let us now prove the shape differentiability of $J$ given by 
\eqref{eq:cost_func} at all open sets $\Omega\subset D$. At first we need a lemma:
\begin{lemma}\label{lemma:phit}
    Let $D\subseteq \R^d$ be open and bounded and suppose $X\in \ac C^1(\overbar D, \R^d)$.
    \begin{itemize} 
        \item[(i)] We have
        \begin{align*}
        \frac{ \partial \Phi_t - I}{t} \rightarrow  \partial X \quad 
        &\textsf{ and }\quad   \frac{ 
            \partial \Phi_t^{-1} - I}{t} \rightarrow  - \partial X && 
        \textsf{ strongly in } C(\overline D, \R^{d,d})\\
        \frac{ \det(\partial \Phi_t) - 1}{t} \rightarrow & \divv(X) && \textsf{ strongly in } C(\overline D).
        \end{align*}
        \item[(ii)] 	For all open sets $\Omega \subseteq D$ and all $\varphi\in W^1_\mu(\Omega)$, $\mu \ge 1$, we have 
        \begin{align}
        \frac{\varphi\circ \Phi_t  - \varphi}{t} \rightarrow & \nabla \varphi \cdot X
        && \textsf{ strongly in }  L_\mu(\Omega).
        \end{align}
    \end{itemize}
\end{lemma}

Now we can prove the shape differentiablity of $J$.
\begin{theorem}\label{thm:shape_d_transmission}
    Let $X\in \ac C^1(\overbar D, \R^2)$ be given and denote by $\Phi_t$ the $X$-flow. 
    The shape derivative of $J$ given by \eqref{eq:cost_func} at 
    a measurable subset $\Omega \subset D $ is given by
    \ben\label{eq:vol_shape_deri_transmission}
    dJ(\Omega)(X) = \int_D \Sb_1(\Omega,u,p):\partial X + \Sb_0(\Omega,u,p) \cdot X\, dx,	
    \een
    where $u\in \ac H^1(D,\R^2)$ solves the state \eqref{eq:trans_state_lin} 
    and $p\in \ac H^1(D,\R^2)$ solves the adjoint state equation
    \ben\label{eq:trans_adjoint_state_lin}
    \int_{D} \beta_\chi \nabla v \cdot \nabla p \, dx  = - \int_D 2(u -u_d) v \, dx \quad \textsf{ for all } v \in \ac{H}^1(D).
    \een
    and  
    \ben\label{eq:Sb_0_Sb_1}
    \begin{split}
        \Sb_1(\Omega,u,p) & = -\beta_\chi\nabla u\otimes \nabla p- \beta_\chi\nabla p \otimes \nabla u + I (\beta_\chi\nabla u \cdot \nabla p - fp +  |u-u_d|^2) \\
        \Sb_0(\Omega,u,p) & = - p\nabla f - 2(u-u_d)\nabla u_d.  	
    \end{split}
    \een
    If $\Gamma\in C^2$, then $u^\pm,p^\pm\in H^2(\Omega^\pm)$,
    \ben
    -\divv(\Sb_1(\Omega,u,p))+\Sb_0(\Omega,u,p)=0 \quad \textsf{ a.e. in } \Omega\quad\textsf{  and } \quad \textsf{ a.e. in }  D\setminus \overbar\Omega, 	
    \een
    and
    \ben\label{eq:boundary_expression_trm}
    dJ(\Omega)(X) = \int_{\Gamma} \lb\Sb_1(\Omega,u,p)\nu \rb \cdot  X \, ds = \int_{\Gamma} \lb\Sb_1(\Omega,u,p)\nu\cdot \nu\rb \; (X\cdot 
    \nu)\, ds,
    \een
    where $\Sb_1^{\pm}(\Omega,u,p) := (\Sb_1(\Omega,u,p))_{|\Omega^{\pm}}$. 
\end{theorem}
\begin{proof}
    The proof is an adaption of the proof of Proposition 5.2 in 
    \cite{lauraindistributed}. However, let us sketch the ingredients of the proof. At first we consider
    equation \eqref{eq:trans_state_lin} with characteristic function $\chi_{\Omega_t}$, 
    $\Omega_t := \Phi_t(\Omega)$,  
    \ben\label{eq:trans_state_lin_per}
    \int_{D} \beta_{\chi_{\Omega_t}}\nabla u_t\cdot \nabla v \, dx  =  \int_{D} f  
    v \, dx  \quad \textsf{ for all } v \in \ac{H}^1(D).
    \een
    Using $\chi_{\Omega_t} = \chi \circ\Phi_t^{-1}$ and setting 
    $u^t:=u_t\circ \Phi_t$, a change of variables shows \eqref{eq:trans_state_lin_per} is equivalent to
    \ben\label{eq:trans_state_lin_per_2}
    \int_{D} \beta_\chi A(t)\nabla u^t\cdot \nabla v \, dx  =  
    \int_{D} \xi(t) f^t  
    v \, dx  \quad \textsf{ for all } v \in \ac{H}^1(D),
    \een
    where $\xi$ and $A$ are defined in \eqref{eq:notation}.
    Let us introduce the Lagrangian
    \ben\label{eq:lagrangian_G}
    G(t,X,w,v) = \int_D \xi(t)|w - u_d^t|^2 \, dx + \int_{D} \beta_\chi A(t)\nabla w\cdot \nabla v\, dx  -  
    \int_{D} \xi(t) f^t  
    v  \, dx
    \een
    with the definitions
    \ben\label{eq:notation}
    \xi(t):=\det(\partial \Phi_t),\qquad A(t):=\xi(t)\partial \Phi_t^{-1}\partial 
    \Phi_t^{-\top}, \quad f^t :=f\circ\Phi_t, \quad u_d^t = u_d \circ \Phi_t.
    \een
    Thanks to Lemma~\ref{lemma:phit} the derivatiev $\partial_t G(0,v,w)$ exists for all 
    $w,v\in \ac H^1(D,\R^2)$ and is given by
    \ben\label{eq:partial_t_G}
    \partial_t G(0,X,w,v)= \int_D \beta_\chi A'(0)\nabla w\cdot \nabla v + \xi'(0)(|w-u_d|^2 - fv) - \nabla f\cdot X v \;dx.  
    \een
    Now it can be shown that cf. \cite{lauraindistributed,sturm2015shape,MR3374631}
    \ben
    dJ(\Omega)(X) = \dt G(t,X,u^t,p)|_{t=0} = \partial_t G(0,X,u,p),	
    \een
    where $p\in \ac{H}^1(D)$ is the solution of \eqref{eq:trans_adjoint_state_lin}.
    From this and \eqref{eq:partial_t_G} it can be inferred that
    \ben
    dJ(\Omega)(X) = \int_D \Sb_1(\Omega,u,p):\partial X + \Sb_0(\Omega,u,p) \cdot X\, dx	
    \een
    with the definitions 
    \ben
    \Sb_1(\Omega,u,p) = -\beta_\chi\nabla u\otimes \nabla p- \beta_\chi\nabla p \otimes \nabla u + I (\beta_\chi\nabla u \cdot \nabla p - fp +  |u-u_d|^2), 
    \een
    \ben
    \Sb_0(\Omega,u,p) = - p\nabla f - 2(u-u_d)\nabla u_d.  	
    \een
    Now if $\Gamma\in C^2$, then by standard regularity theory we obtain 
    $u^\pm,p^\pm\in H^2(\Omega^\pm)$. Therefore 
    $\Sb_1^\pm(\Omega,u,p)\in W^1_1(\Omega^\pm,\R^{2,2}) $. Thus, Proposition 
    3.3 in \cite{lauraindistributed} shows
    \ben
    dJ(\Omega)(X) = \int_{\Gamma} \lb\Sb_1(\Omega,u,p)\nu \rb \cdot  X \, ds = \int_{\Gamma} \lb\Sb_1(\Omega,u,p)\nu\cdot \nu\rb \; (X\cdot 
    \nu)\, ds
    \een
    and additionally 
    \ben
    -\divv(\Sb_1(\Omega,u,p))+\Sb_0(\Omega,u,p)=0 \quad \textsf{ a.e. in } \Omega\quad\textsf{  and } \quad \textsf{ a.e. in }  D\setminus \Omega. 	
    \een
\end{proof}
\begin{corollary}
    Let $X\in \ac C^{0,1}(\R^2,\R^2)$ with $X=0$ on $\partial D$.
    Then for all measurable $\Omega\subset D$,
    \ben
    \partial_X J(\Omega) :=   \lim_{t\searrow 0} \frac{J((\textsf{id}+tX)(\Omega)) - J(\Omega)}{t} = \int_D \Sb_1(\Omega,u,p):\partial X + \Sb_0(\Omega,u,p) \cdot X\, dx,	
    \een
    where $\Sb_0$ and $\Sb_1$ are given by \eqref{eq:Sb_0_Sb_1}.
\end{corollary}
\begin{proof}
    It is not difficult to check that for $X\in 
    C^{0,1}(\R^2,\R^2)$ with $X=0$ on $\partial D$ the 
    transformation $T_t(X):\R^2\rightarrow \R^2, x\mapsto x+t X(x)$ is bi-Lipschitz for all $t < 1/L$, 
    where $L$ denotes the Lipschitz constant of $X$.  
    Then we notice that items $(i)$ and $(ii)$ of Lemma~\ref{lemma:phit} also hold 
    when we replace $\Phi_t$ by $T_t(X)$. As a consequence $\partial_t G(0,X,w,v)$ exists 
    also in this case for all $v,w\in \ac{H}^1(D)$. 
    Now we can follow the lines of the proof 
    of Theorem~\ref{thm:shape_d_transmission} only replacing the flow 
    $\Phi_t$ by the transformation $T_t(X)$. 
\end{proof}
\begin{remark}
    Notice that the transformation $T_t(X)$ is the $\tilde X$-flow of the time-dependent vector field 
    $\tilde X(x,t):=X\circ (\textsf{id}+tX)^{-1}(x)$, that is, $\Phi_t^{\tilde X}=T_t(X)$; cf. \cite[Chapter 4]{MR2731611}. It is important to note 
    that $dJ(\Omega)(X)$ may not be well-defined for all $X\in \ac 
    C^{0,1}(\R^2,\R^2)$ with $X=0$ on $\partial D$ as $t\mapsto 
    \partial \Phi_t^X$ is not differentiable in $C(\overbar D,\R^{2,2})$ at $t=0$. 
\end{remark}

\subsection{Discretised shape derivative}\label{subsec:discretised}
In the recent article \cite{MR2642680} the relationship between the analytical and discretised 
shape derivative has been studied for a specific model problem. The rigorous 
numerical analysis was carried out in \cite{MR3348199}.
Here we want to recast these results in terms of our tensor representation of 
the shape derivative. 

\subsubsection*{Finite element approximation}
Suppose that $D$ is a polygonal set.  Let $V_h^k$, $k\ge 1$, be the space defined in \eqref{eq:V_h_k}.
Then the finite element approximation of state equation 
\eqref{eq:trans_state_lin} and the adjoint state equation 
\eqref{eq:trans_adjoint_state_lin} reads:
\ben\label{eq:approx_state_adjoint}
\begin{split}
    \int_{D} \beta_\chi\nabla u_h\cdot \nabla \varphi \, dx  & =  \int_{D} f  
    \varphi  \, dx  \quad \textsf{ for all } \varphi \in V_h^k \\ 
    \int_{D} \beta_\chi \nabla \varphi \cdot \nabla p_h \, dx  & = - \int_D 2(u_h -u_d)\varphi \, dx \quad \textsf{ for all } \varphi \in V^k_h.
\end{split}
\een
With the discretised state and adjoint state equation the discretised version 
of the shape derivative given by \eqref{eq:vol_shape_deri_transmission} reads 
\ben\label{eq:vol_discrete}
dJ_{h}^{vol}(\Omega)(X) = \int_{D} \Sb_1^h:\partial X + \Sb_0^h \cdot X\, dx,	
\een
with 
\ben\label{eq:tensor_Sb_1_h}
\Sb_1^h := \Sb_1(\Omega,u_h,p_h) = -\beta_\chi\nabla u_h\otimes \nabla p_h- \beta_\chi\nabla p_h \otimes \nabla u_h + I (\beta_\chi\nabla u_h \cdot \nabla p_h - fp_h +  |u_h-u_d|^2),
\een
\ben\label{eq:tensor_Sb_0_h}
\Sb_0^h := \Sb_0(\Omega,u_h,p_h) = - p_h\nabla f - 2(u_h-u_d)\nabla u_d.  	
\een
%At this stage it should be pointed out that in case $k=1$ and if the vector field belongs to 
%$X_h\in V_h^1\times V_h^1 $,  then discretise and optimise compute with each 
%other, that is, 
%$$ dJ_{h}^{vol}(\Omega)(X_h)  = \lim_{t\searrow 0} \frac{J_h(\Phi_t(\Omega)) - J_h(\Omega)}{t},	$$
%where
%$$ J_h(\Omega) = \int_{D}|u_h-u_d|^2\; dx. 	$$
%This readily follows from a close inspection of the proof of Theorem~\ref{thm:shape_d_transmission}.
%However, in case $k\ge 2$ thi is no longer true. 
\subsubsection*{Comparison of discretised domain and boundary expression}
At first we observe that the discretised volume expression $dJ^{vol}_h(\Omega)$ 
given by \eqref{eq:vol_discrete} does not have the
nice property to be supported on the boundary $\Gamma$ even for smooth 
vector fields: 
there exists $X\in \ac C^{0,1}(\overbar{D},\R^2)$ so that 
$dJ_{h}^{vol}(\Omega)(X)\ne 0 $
and there exists at least one point $x$ in $\Omega \cup
(D\setminus \Omega) $, so that
$
-\divv((\Sb_1^h)^{\pm}) + (\Sb_0^h)^{\pm}  \ne  0.  
$
Therefore $dJ^{vol}_h(\Omega)$ is not equivalent to its discretised boundary 
counterpart
\ben
dJ_h^{bd,1}(\Omega)(X) := \int_{\Gamma} \lb \Sb_1^h\nu\cdot\nu \rb  (\nu\cdot 
X) \; ds
\een
for $X\in \ac C^1(\overbar D,\R^2)$.
Recall that the boundary expression of $dJ(\Omega)$ in the continuous case was computed in 
\eqref{eq:boundary_expression_trm} and reads 
\ben
dJ(\Omega)(X) = \int_{\Gamma} \lb\Sb_1\nu\rb\cdot \nu \; (X\cdot \nu)\, ds. 
\een
Moreover, we have the following equivalence (cf. \cite{lauraindistributed}) 
\ben
\int_{\Gamma} \lb\Sb_1 \nu\cdot \nu\rb \; (X\cdot \nu)\, ds  = \int_{\Gamma} 
\lb \Sb_1\nu \rb \cdot X \, ds, \quad X\in \ac C^1(\overbar D,\R^2).    
\een
Accordingly there is another possible way to discretise the boundary expression:
\ben
dJ_h^{bd,2}(\Omega)(X) := \int_{\Gamma} \lb \Sb_1^h\nu \rb  \cdot X \; ds,
\een
which neither conicides with $dJ_h^{bd,1}(\Omega)(X)$ nor with 
$dJ_h^{vol}(\Omega)(X)$. 
In fact we can prove by partial integration that the three previously 
introduced discretisations of the shape derivative are related. 

Recall 
that $\mathcal E_h$ denotes the edges of the triangulation $\mathcal T_h$ of 
$D$.  
\begin{lemma}
    Let $\Omega\subset D$ be a polygonal domain, so that, 
    $\partial \Omega=\cup_{\substack{E\in \mathcal E_h \\ \mathcal E_h\cap \Gamma\neq \emptyset}}\{E\}$
    .  We have for all $X\in \ac C^1(\overbar D,\R^2)$ 
    \begin{align} dJ_{h}^{vol}(\Omega)(X)  = &  
    dJ_h^{bd,2}(\Omega)(X)  +   \sum_{K\in \mathcal T_h}    \int_{K} (-\divv(\Sb_1^h) + \Sb_0^h) 
    \cdot X\, dx \\
    &  +    \sum_{\stackrel{E\in \mathcal E_h}{E\not\subset \Gamma }}\int_{E} \lb \Sb_1^h \nu_{ E} \rb \cdot  X \, ds, 
    \end{align}
    or equivalently 
    \begin{equation}
    \begin{split} dJ_{h}^{vol}(\Omega)(X)  = &  dJ_h^{bd,1}(\Omega)(X)  +  \sum_{K\in \mathcal T_h}    \int_{K} (-\divv(\Sb_1^h) + \Sb_0^h) 
    \cdot X\, dx  \\
    &  +    \sum_{\stackrel{E\in \mathcal E_h}{E\not\subset \Gamma }}\int_{E} \lb \Sb_1^h \nu_{E} \rb \cdot  X \, ds  + \int_{\Gamma} \lb\Sb_1^h \nu_\Gamma  - \Sb_1^h 
    \nu_\Gamma \cdot \nu_\Gamma\rb  \cdot  X \, ds. 
    \end{split}
    \end{equation}
\end{lemma}
\begin{proof}
    At first notice that for all $K\in \mathcal T_h$ we have
    $(\Sb_1^h)_{|K} \in C^\infty(\overbar{K}). $ 
    Hence it follows by partial  integration on each element $K\in \mathcal T_h$,
    \begin{align*} dJ_{h}^{vol}(\Omega)(X) & = \int_{D} \Sb_1^h:\partial X + 
    \Sb_0^h \cdot X\, dx \\
    & = \sum_{K\in \mathcal T_h} \int_K \Sb_1^h:\partial X + \Sb_0^h \cdot X\, dx\\
    & = \sum_{K\in \mathcal T_h}    \int_{K} (-\divv(\Sb_1^h) + \Sb_0^h) 
    \cdot X\, dx+ \sum_{K\in \mathcal T_h}    \int_{\partial K} \Sb_1^h\nu_K \cdot  
    X \, ds. 
    \end{align*}
    Now the result follows at once from 
    \ben  \sum_{K\in \mathcal T_h}    \int_{\partial K} \Sb_1^h\nu_K \cdot  
    X \, ds = \sum_{\stackrel{E\in \mathcal E_h}{E\not\subset \Gamma }}   \int_{E} \lb\Sb_1^h\nu_{E}\rb \cdot  
    X \, ds +   \int_{\Gamma} \lb\Sb_1^h\nu\rb \cdot  
    X  \, ds \een
    and by rearranging. 
\end{proof}
Finally we note that 
$dJ_h^{bd,1}(\Omega)(X) \neq dJ_h^{bd,2}(\Omega)(X)$.  
It is clear that in some sense
\ben
dJ(\Omega)(X) \approx dJ_h^{bd,1}(\Omega)(X) \approx dJ_h^{bd,2}(\Omega)(X) \approx dJ_h^{vol}(\Omega)(X).	
\een
For a rigorous error analysis and more details we refer to \cite{MR3348199} and also \cite{P14_562}.

\section{Numerics}
\label{sec:experiments}

This section is devoted to the practical demonstration of vvRKHS based shape optimisation.
The numerical experiments with two different kernels show that this approach is a very efficient and robust numerical tool. We compare these kernel methods with two other typically used 
gradients, the Euclidean gradient and the $H^1$ gradient, both computed in the 
conforming P1 finite element space.
All methods are applied to the transmission problem \eqref{eq:trans_state_lin}.

\subsection{Numerical setting and algorithm}
The subsequent computations are carried out on the domain $D=(0,1)^2$ which is 
in accordance with our assumption in the previous section. In all test cases the set 
$\Omega\subset D$ is assumed to be polygonal.
The initial mesh consists of 900 elements as shown in Figure~\ref{fig:shape-setup} and the interface of the initial circular shape is discretised with 100 equidistant vertices.

In the following, $J$ is the shape function defined in \eqref{eq:cost_func} with shape derivative at
$\Omega\subset D$ (cf. \eqref{eq:vol_discrete}),
\ben
dJ^h_{vol}(\Omega)(X) =  \int_{D} \Sb_1^h(\Omega):\partial X + \Sb_0^h(\Omega) \cdot X\, dx. 
\een
Here $\Sb_0^h(\Omega)=\Sb_0^h(\Omega,u_h,p_h)$ and $\Sb_1^h=\Sb_1^h(\Omega,u_h,p_h)$ are defined in \eqref{eq:tensor_Sb_1_h} and \eqref{eq:tensor_Sb_0_h}, respectively.  They 
are approximations of $\Sb_0(\Omega,u,p)$ and $\Sb_1(\Omega,u,p)$ given by \eqref{eq:Sb_0_Sb_1}.
The approximations $u_h$ and $p_h$ of the adjoint state and the state are given by \eqref{eq:approx_state_adjoint} where we choose 
$k=1$.

\subsubsection*{Standard gradient algorithm}
Suppose some Hilbert space 
$\mathcal V_N(D, \R^2)\subset H^1(D,\R^2)$.
The gradient $\nabla J^h(\Omega)$ of $J$ is computed by 
\ben\label{eq:gradient_discrete}
(\nabla J^h(\Omega), X)_{\mathcal V_N(\mathcal X, \R^2)} =  \int_{D} \Sb_1^h:\partial X + \Sb_0^h \cdot X\, dx\quad \forall X \in \mathcal V_N(\mathcal X, \R^2) 
\een
The basic optimisation algorithm can be described as follows:\\[0.5ex]

\begin{algorithm}[H]
    
    %  \SetLine % For v3.9
    
    %\SetAlgoLined % For previous releases [?]
    
    \KwData{Let $n=0$, $\gamma>0$ and $N\in \N$ be given. Initialise domain $\Omega_0\subset D$, time $t_n=0$.}
    
    %\KwResult{how to write algorithm with \LaTeX2e }
    
    initialization\;
    
    \While{ $n  \le  N $}{
        
        1.) solve \eqref{eq:gradient_discrete} to obtain  $ \nabla J^h = \nabla J^h(\Omega_{n})$\;
        2.) decrease $t > 0$ until $ J^h((\textsf{id}-t\nabla J^h)(\Omega_n)) < J^h(\Omega_n) $\\ 
        \qquad and set $\Omega_{n+1} \gets (\textsf{id}-t\nabla J)(\Omega_n) $\;
        \eIf{ $J^h(\Omega_n) - J^h(\Omega_{n+1}) \ge \gamma (J^h(\Omega_0) - J^h(\Omega_1))$ }{            
            step accepted: continue program\; 
        }{
        no sufficient decrease: quit\;
    }            
    increase $n \gets n+1$\;
}
\caption{Standard algorithm} \label{eq:algo_euler}
\end{algorithm}

\subsubsection*{Variable metric gradient algorithm} 
Let $\mathcal H^\sigma(D,\R^2)$ be the vvRKHS defined by the radial 
kernel $K(x,y)= \phi_\sigma(|x-y|^2)I$, where we choose 
$\phi$ to be (recall $\phi_\sigma(r):=\phi(r/\sigma)$)
\begin{enumerate}
    \item $\phi_1(r) := e^{r}$
    \item $\phi_2(r) := (1-r)_+^4(4r+1)$.
    %     \item $\phi_3(r) := \textsf{sinc}(r)$. 
\end{enumerate} 
Notice that the corresponding RKHS $\mathcal H(D,\R^2)$ is infinite dimensional, depends on $\sigma$ and 
the gradient $\nabla^\sigma J(\Omega)$  of $J$, defined in \eqref{eq:formula_radial_kernel}, in this space also depends on 
$\sigma$. We define the discretised gradient $\nabla^\sigma J^h(\Omega)$ via 
\ben\label{eq:gradient_discrete_sigma}
\nabla^\sigma J^h(\Omega)(y) :=\int_{D} \left(\phi_\sigma(|x-y|^2) 
\Sb_0^h(x)+\frac{2}{\sigma}\phi_\sigma'(|x-y|^2) \Sb_1^h(x) (x-y)  \right)\ dx. 
\een
Here, $\Sb_0^h$ and $\Sb_1^h$ are approximations of $\Sb_0$ and $\Sb_1$ and specified 
for our transmission problem below. It should be emphasised that the gradient 
does not necessarily vanish on $\partial D$. 

We now have gathered all ingredients to state the improved variable metric algorithm.\\[1ex]
\begin{algorithm}[H]
    
    %  \SetLine % For v3.9
    
    %\SetAlgoLined % For previous releases [?]
    
    \KwData{Let $n=0$, $\gamma>0$, $\sigma>0$ and $N\in \N$ be given. Initialise domain $\Omega_0\subset D$ and time $t_n=0$.}
    
    %\KwResult{how to write algorithm with \LaTeX2e }
    
    initialization\;
    
    \While{ $n  \le  N $}{
        1.) solve \eqref{eq:gradient_discrete} to obtain  $ \nabla^\sigma J^h = \nabla^\sigma J^h(\Omega_{n})$\;
        2.) decrease $t >0$ until 
        $ J^h((\textsf{id}-t\nabla^\sigma J^h )(\Omega_n)) < J^h(\Omega_n) $\\ 
        \qquad and set $\Omega_{n+1} \gets (\textsf{id}-t\nabla^\sigma J^h)(\Omega_n)$\;
        \eIf{ $J^h(\Omega_n) - J^h(\Omega_{n+1}) \ge \gamma (J^h(\Omega_0) - J^h(\Omega_1))$ }{
            step accepted: continue program\;   
        }{
        decrease $\sigma \gets q \sigma$, $q\in (0,1)$\;
    }
    increase   $n \gets n+1$\;     
}

\caption{Variable metric algorithm.} \label{eq:algo_euler-2}

\end{algorithm}

\begin{remark}
    Algorithm \ref{eq:algo_euler-2} represents a new type of algorithm 
    for shape optimisation since it includes a change of the metric during the optimisation 
    process.
\end{remark}

\subsubsection*{Numerical tests}
In Figure~\ref{fig:RK-shapes}, the results of Algorithm~\ref{eq:algo_euler-2} with
parameters $\sigma=10$, $\gamma=10^{-2}$, $q=0.5$, $f=1$, $\beta_1=1$, $\beta_2=0.5$ and the
gradient defined in \eqref{eq:gradient_discrete_sigma} are depicted for some selected iteration steps.
In the left picture the reproducing kernel associated to $\phi_1$ is chosen and in the
right picture, the kernel associated to $\phi_2$ is employed. The inital shape 
is a circle with radius 0.1 located in the left lower corner with center (0.15,0.15), see Figure~\ref{fig:shape-setup}. 
The optimal shapes are two discs located at (0.65,0.35) and (0.7,0.5) with radii 0.2 and 0.1, respectively.
They are thus located in the upper right corner of the domain and intersect each other. The reference 
function $u_d$ is the solution of the tranmission problem on the optimal domain depicted in Figure~\ref{fig:shape-setup} (right).

The evolutions of the shapes are quite similar but a closer inspection reveals that they are in fact not 
identical. As predicted, for initially large $\sigma$, the shape is only translated but not changed otherwise.
After several iterations, the location of the 
optimal shape is reached and $\sigma$ is successively reduced which enables the subsequent deformation of the shape. 
Eventually, the final shape is very well reconstructed although the initial shape was place quite far away from the optimum.
Additionally, Figure~\ref{fig:RK2-meshes} illustrates the computing mesh with $\phi_2$ for some iterations.
The convergence history for the two examined kernels is depicted in Figure~\ref{fig:errors} (left).

\begin{figure}
    \centering
    \includegraphics[width=0.49\linewidth]{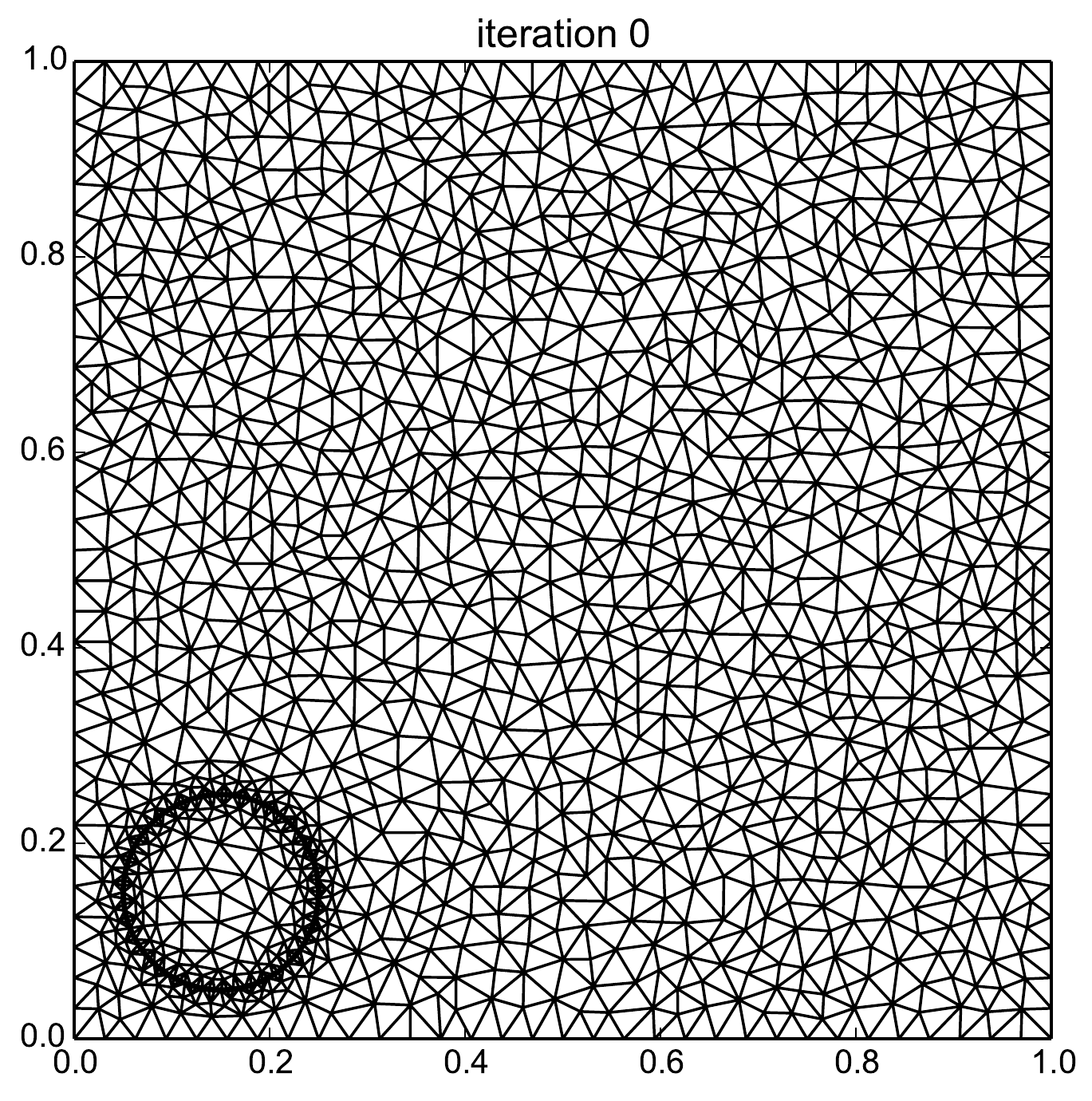}
    \includegraphics[width=0.49\linewidth]{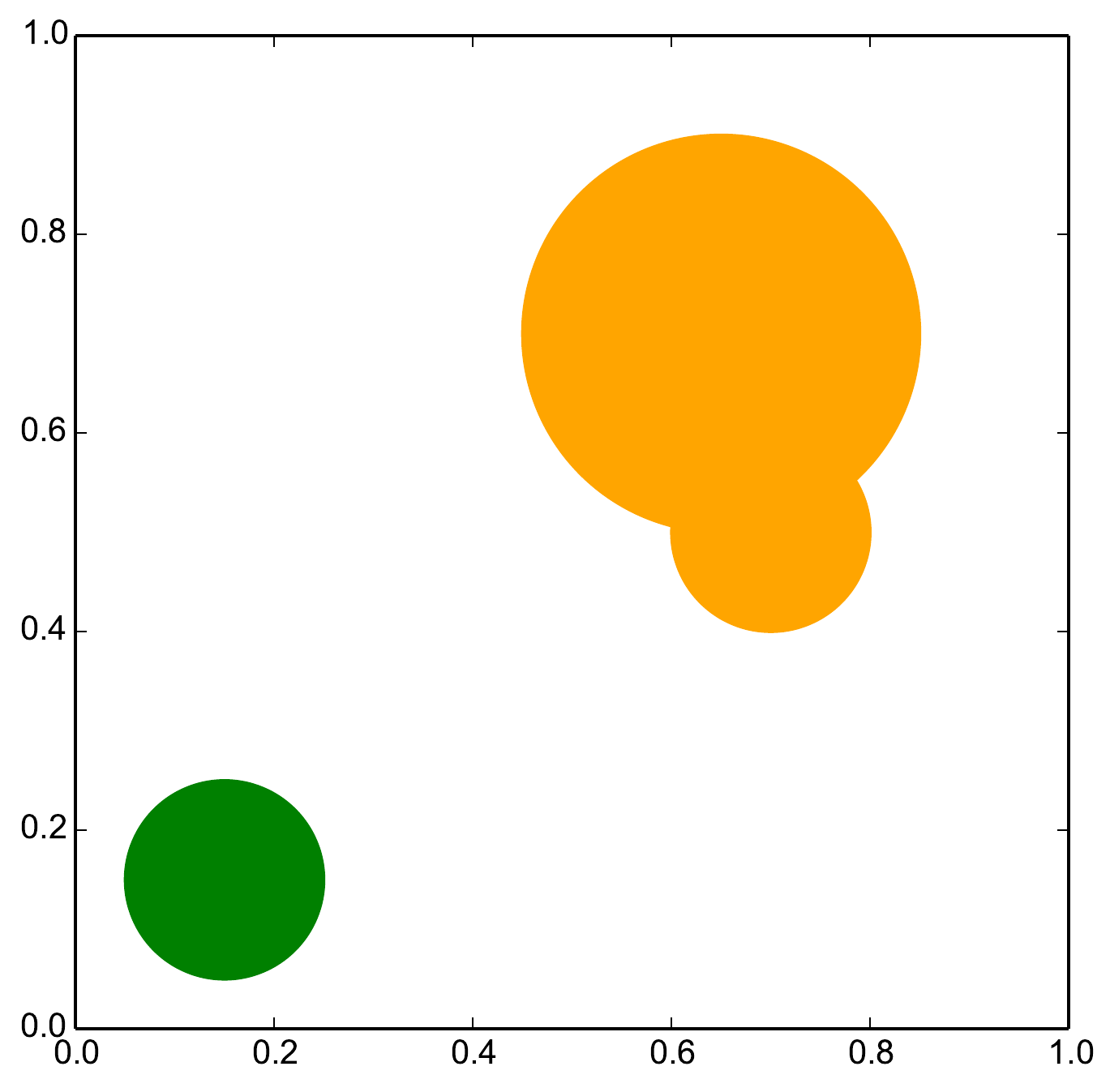}
    \caption{Inital mesh (left) and domain setup (right) with initial shape (bottom left) and optimal shape (top right).}
    \label{fig:shape-setup}
\end{figure}

In Figure~\ref{fig:H1-shapes}, results of Algorithm~\ref{eq:algo_euler} with 
the $H^1$ metric (left) and the Euclidean metric (right) are depicted. 
The gradient in the Euclidean metric is given by (cf. \eqref{eq:gradient_eulid})
\ben
\nabla J^h(\Omega) = \sum_{k=1}^{2N} dJ^h_{vol}(\Omega)(v^k) v^k
\een
and the $H^1$ gradient is defined as the solution of the variational problem
\ben
(\nabla J^h(\Omega), \varphi)_{H^1} = dJ^h_{vol}(\Omega)(\varphi) \quad \text{ for all } \varphi\in \mathcal V_{N_h}(\mathcal X,\R^2),
\een
where the space $\mathcal V_{N_h}(\mathcal X,\R^2)$ is given by \eqref{eq:V_N_h} with $\mathcal X=D$.
The inital shape is now placed very close to the optimal shape and even overlaps it. The reason is 
that both gradient methods, the Euclidean and the $H^1$, are not able to perform
large shape deformation and do do not converge when the inital shape is too far away.
For the Euclidean metric to converge, the initial shape actually has to lie basically inside the optimal shape.

Opposite to this, it poses no problem for our novel variable metrics algorithm which proves to be much more robust
in practise as demonstrated before.
We also point out that the reconstructions in Figure~\ref{fig:H1-shapes} 
are not as good as the previous ones in Figure~\ref{fig:RK-shapes}.
The convergence history for the $H^1$ and the Euclidean metric based optimisations is depicted in Figure~\ref{fig:errors} (right).

\begin{figure}
    \includegraphics[width=0.5\linewidth]{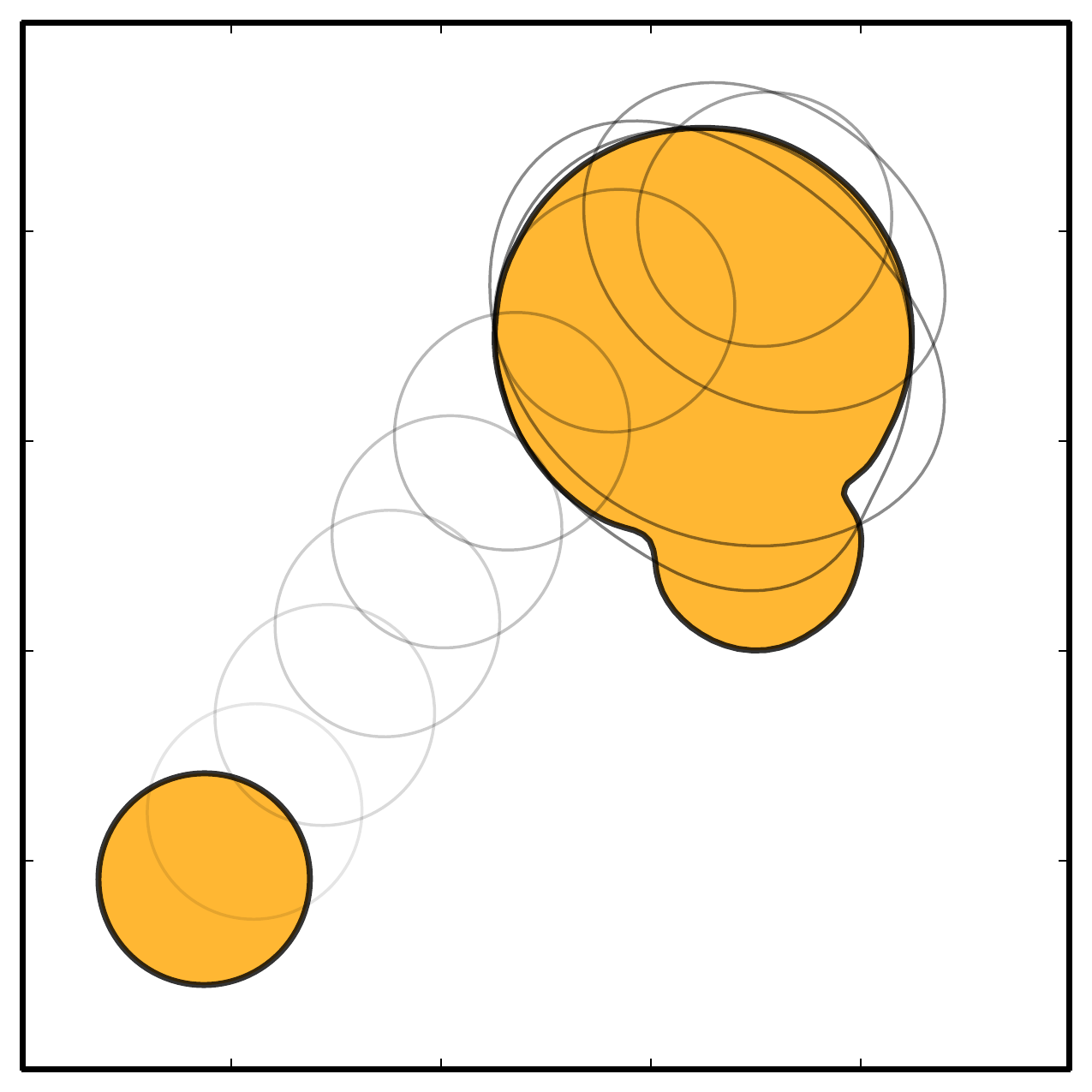}
    \includegraphics[width=0.5\linewidth]{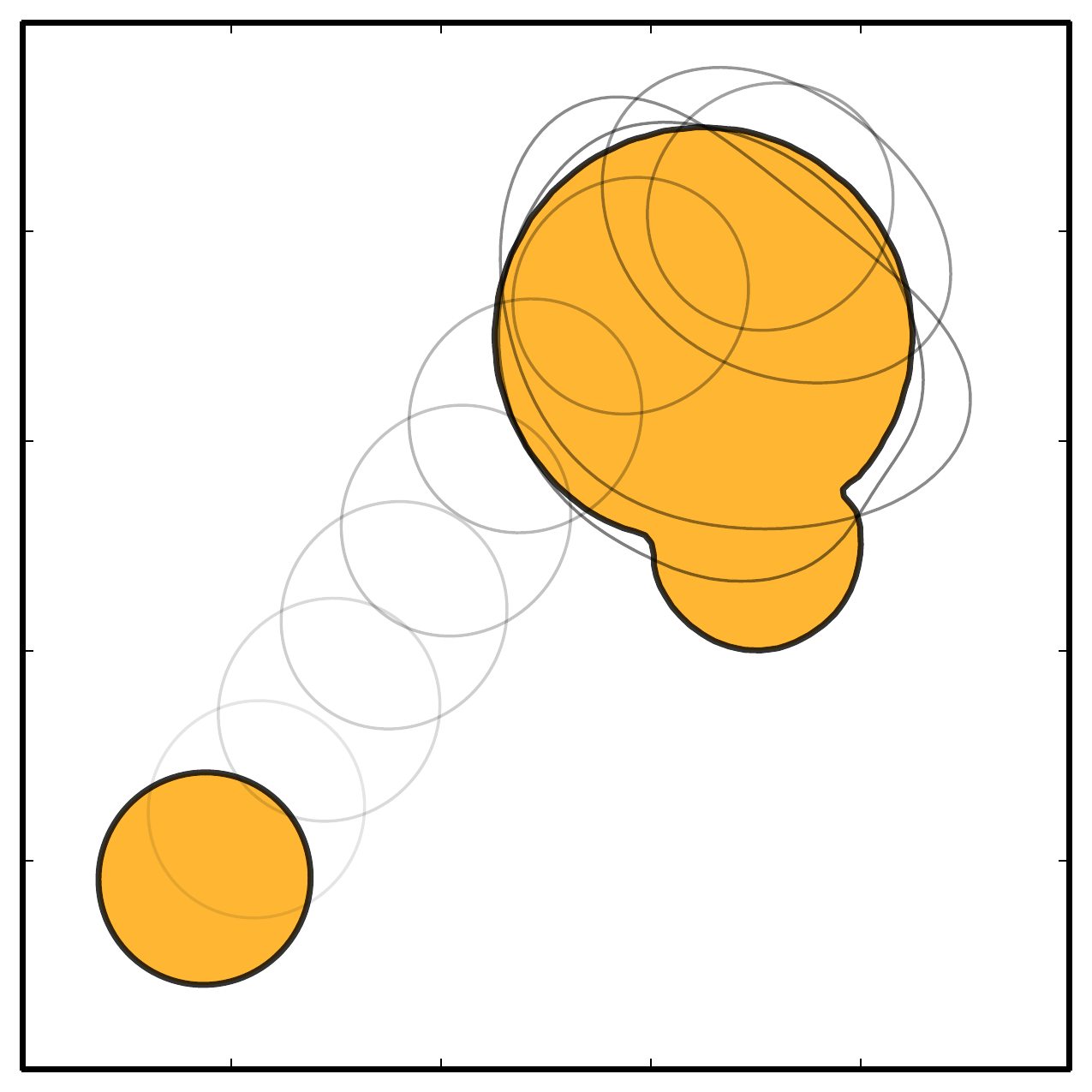}
    \caption{Shape progress for vvRKHS based optimisation with $\phi_1$ (left) and $\phi_2$ (right) for iterations 2, 5, 8, 11, 14, 17, 24, 30, 35, 50.}
    \label{fig:RK-shapes}
\end{figure}

\begin{figure}
    \includegraphics[width=0.24\linewidth]{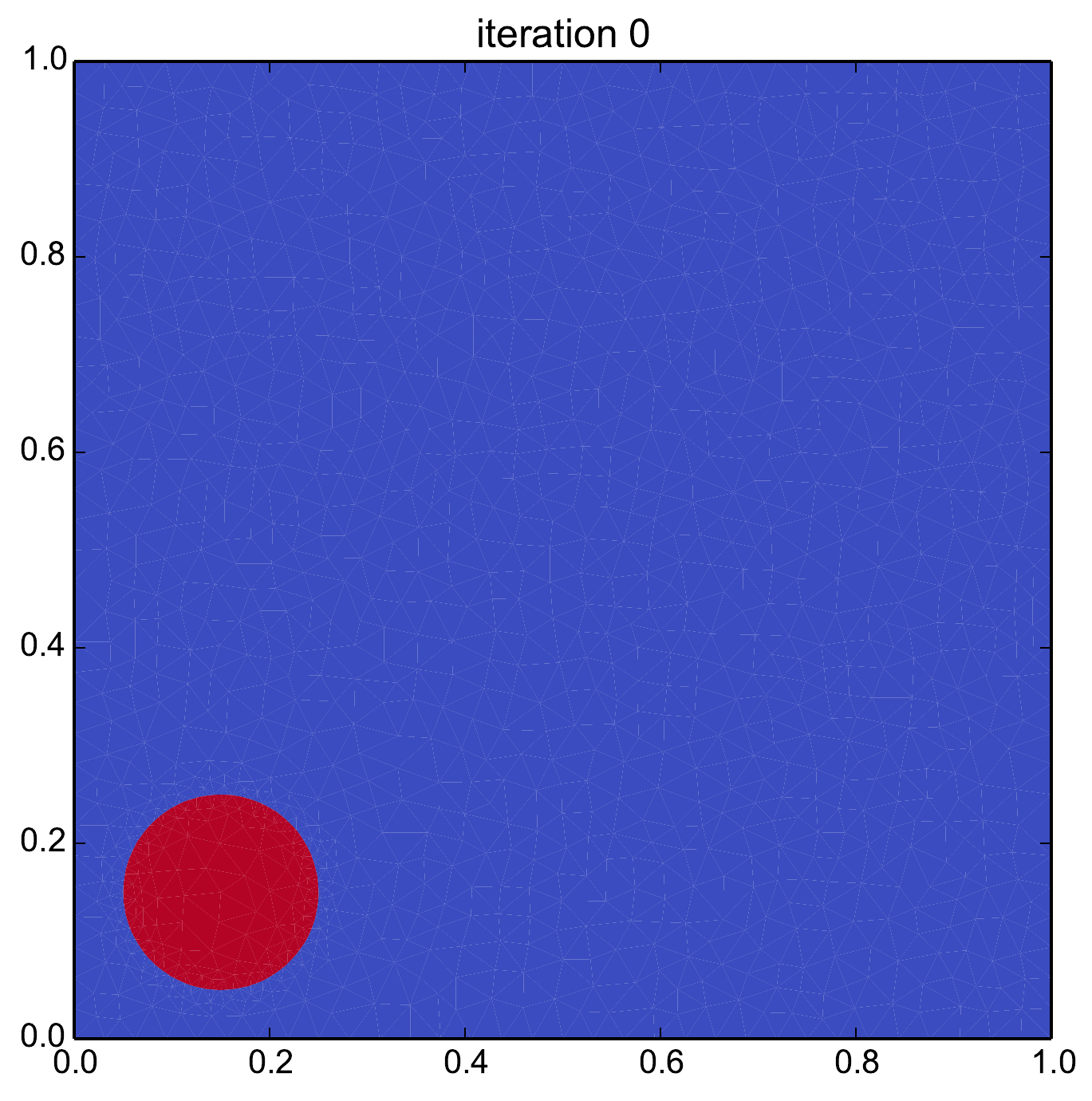}
    \includegraphics[width=0.24\linewidth]{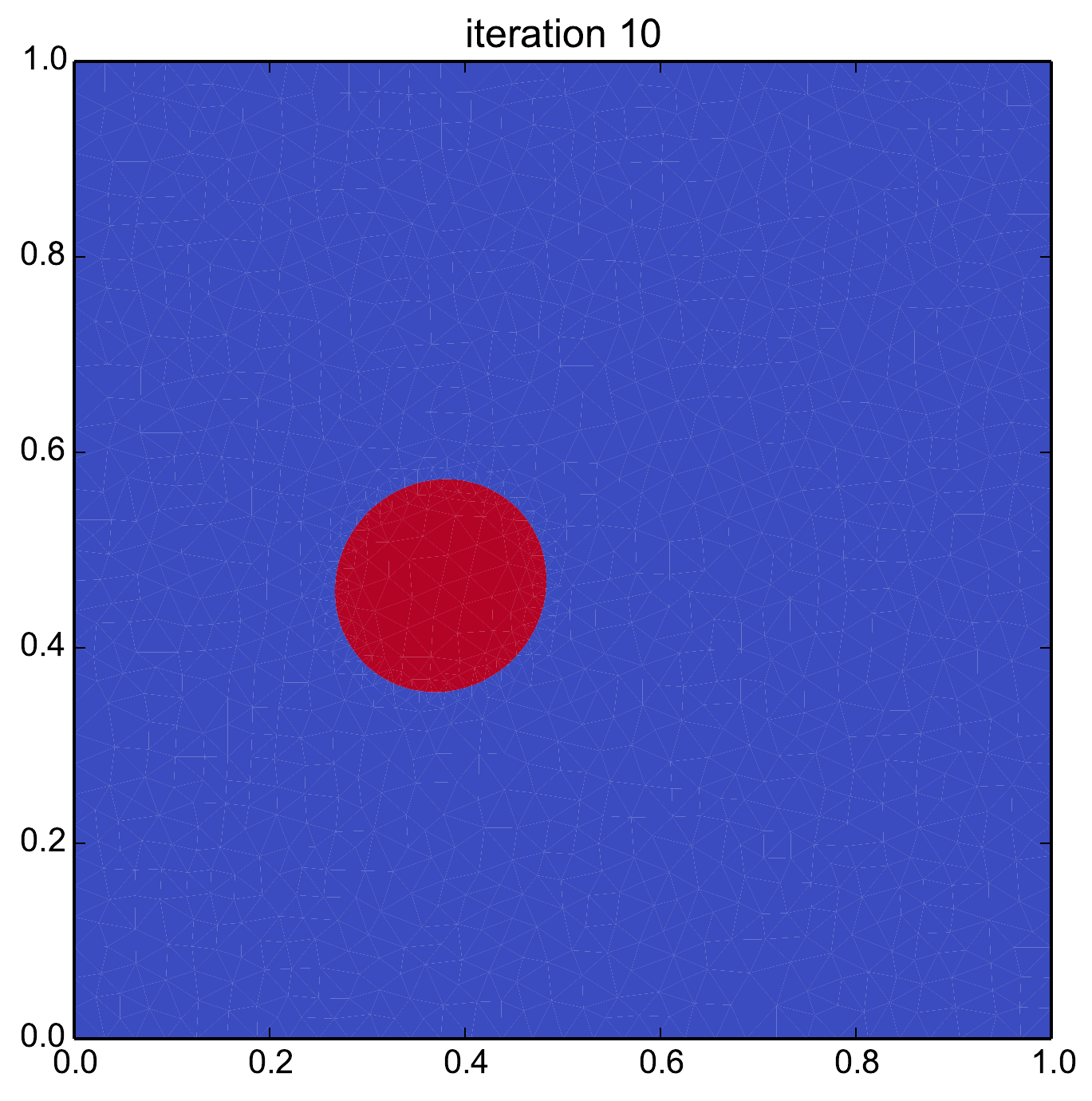}
    \includegraphics[width=0.24\linewidth]{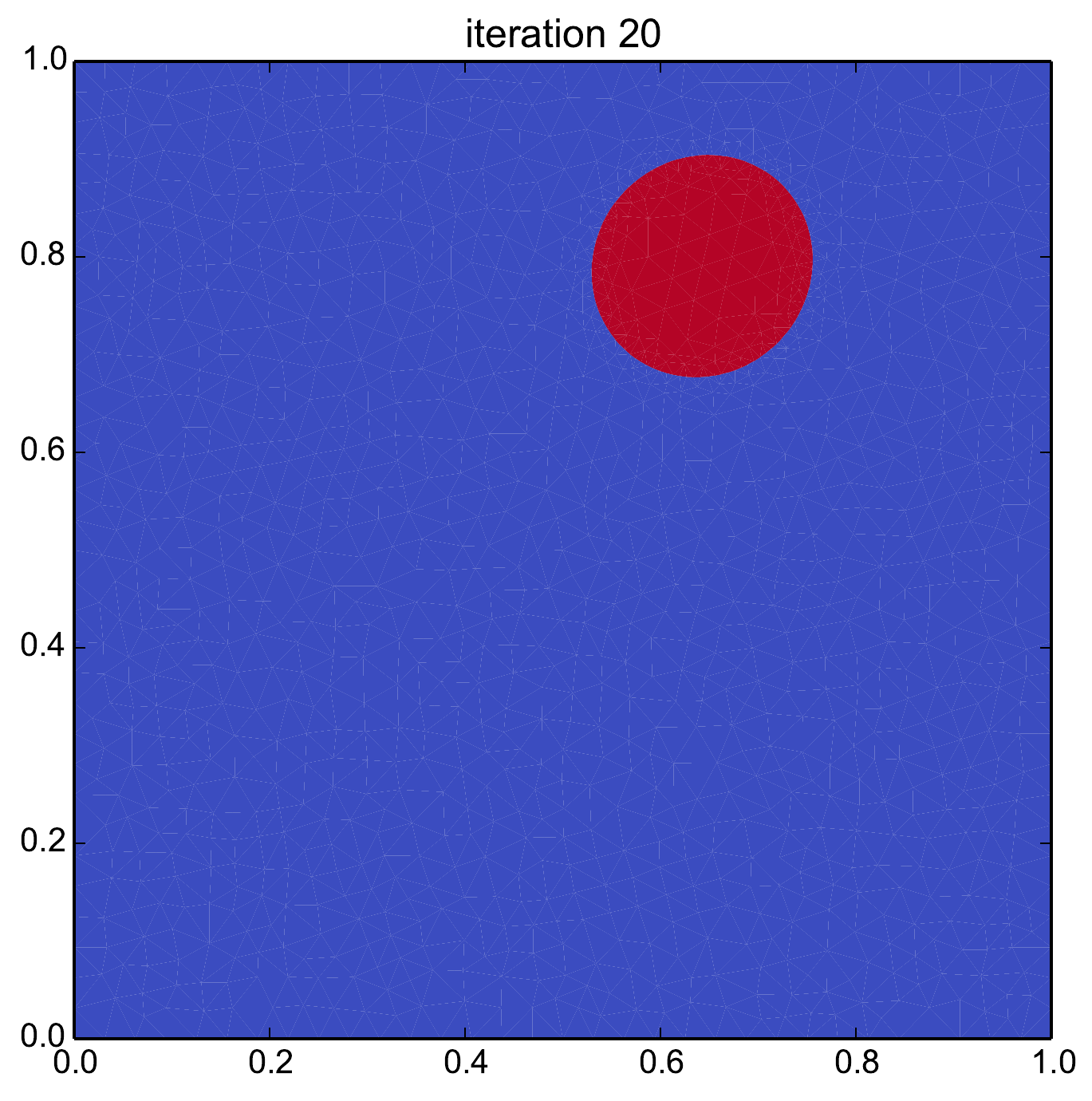}
    \includegraphics[width=0.24\linewidth]{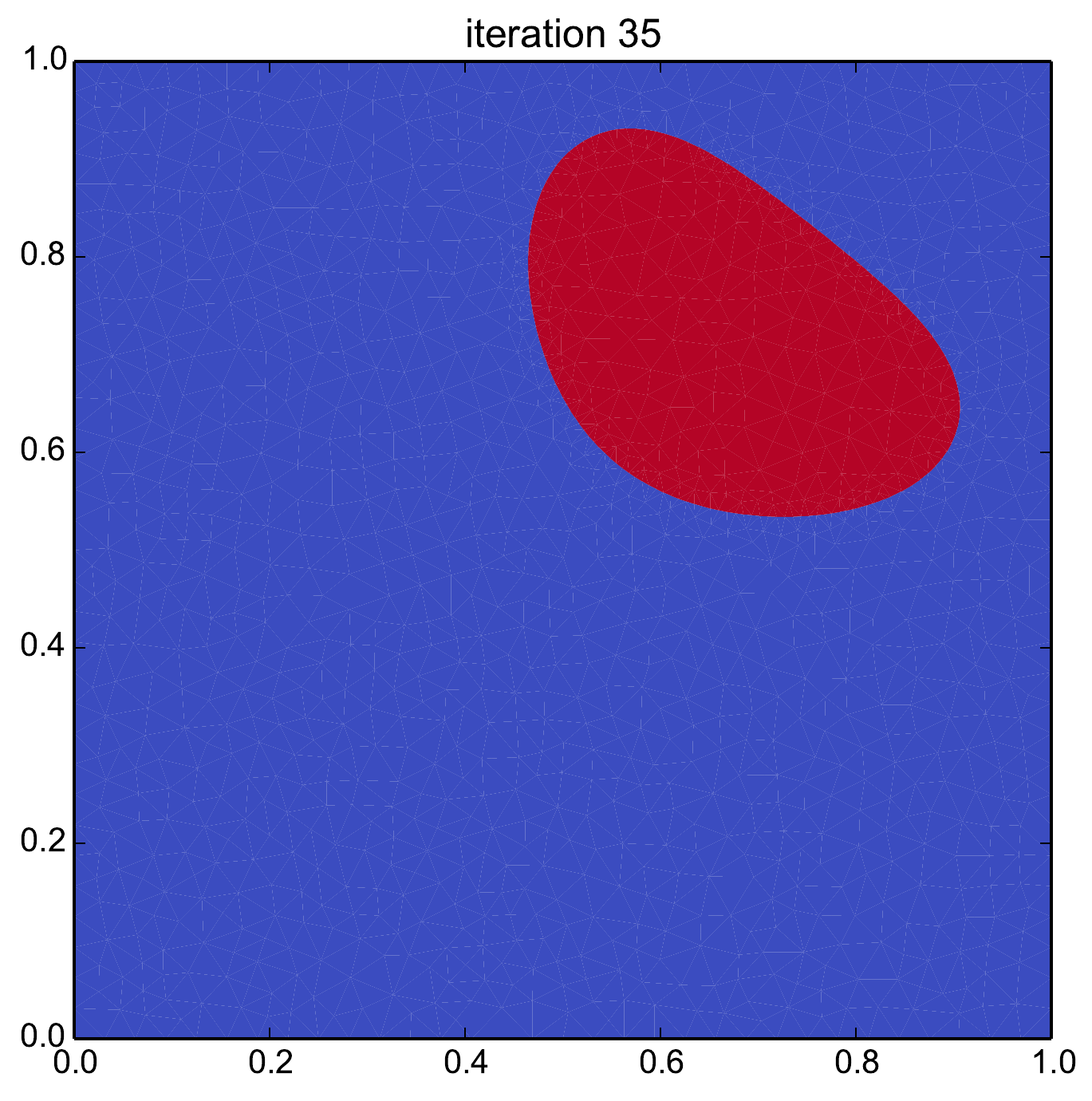}
    \includegraphics[width=0.24\linewidth]{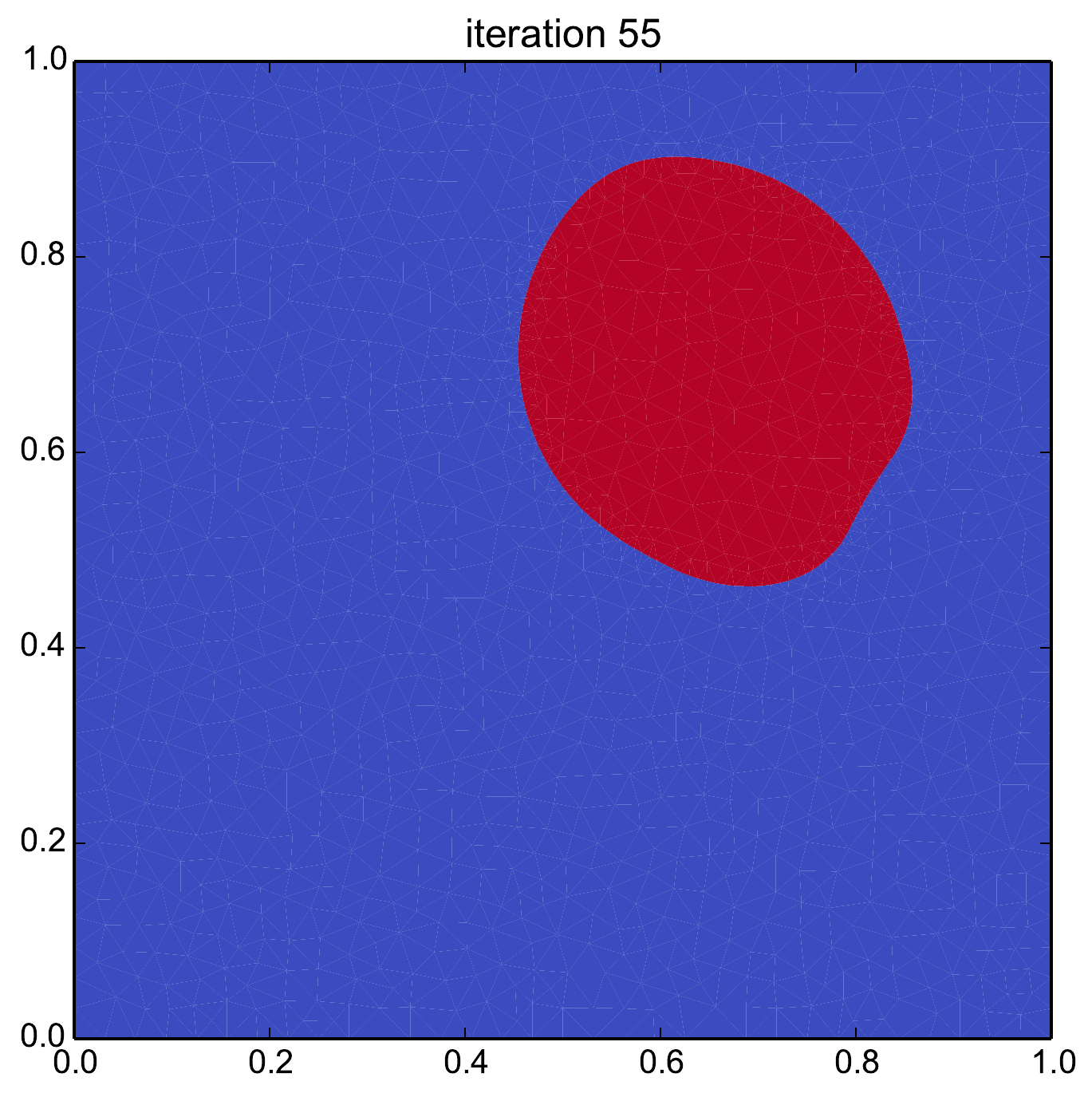}
    \includegraphics[width=0.245\linewidth]{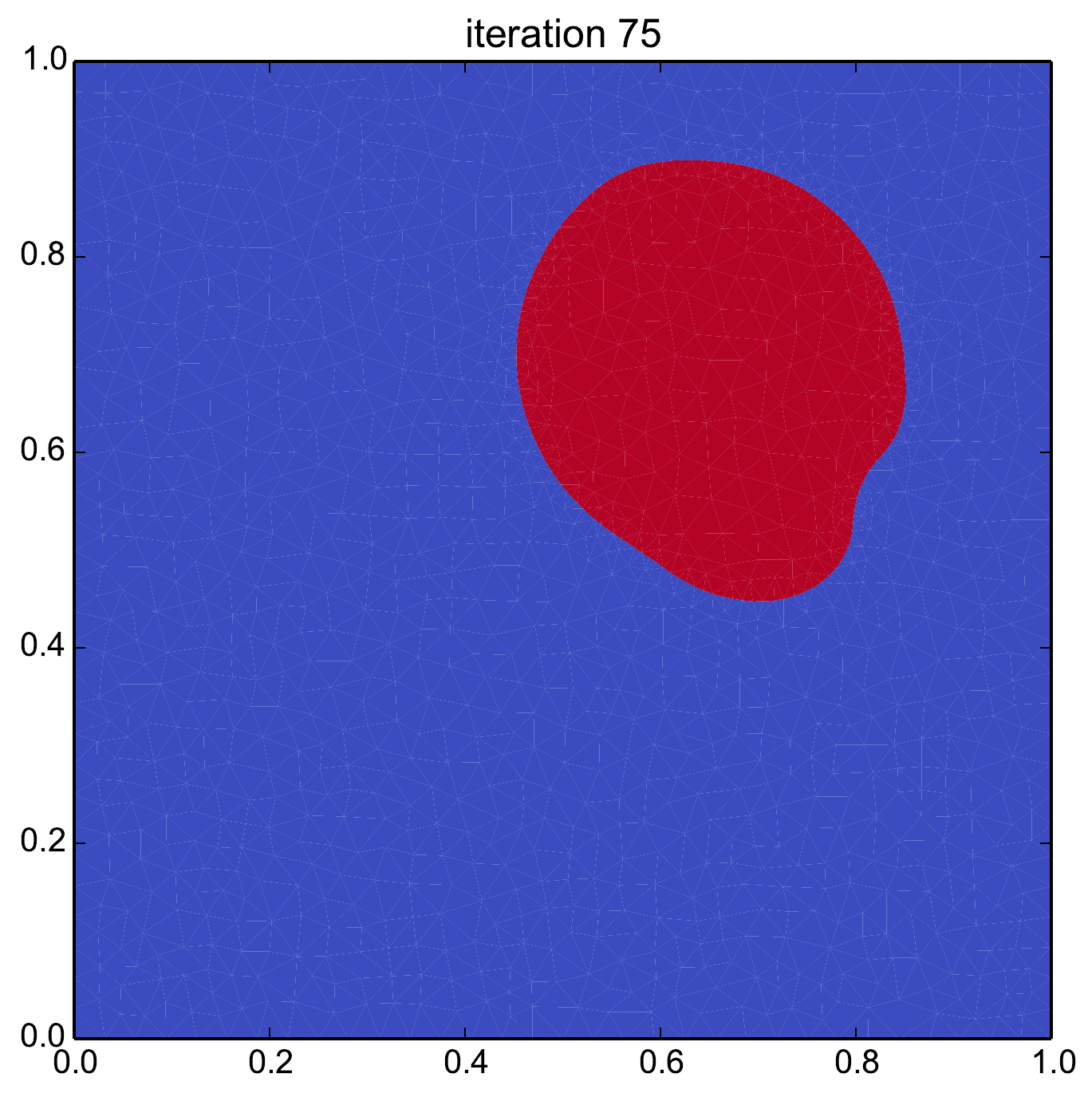}
    \includegraphics[width=0.245\linewidth]{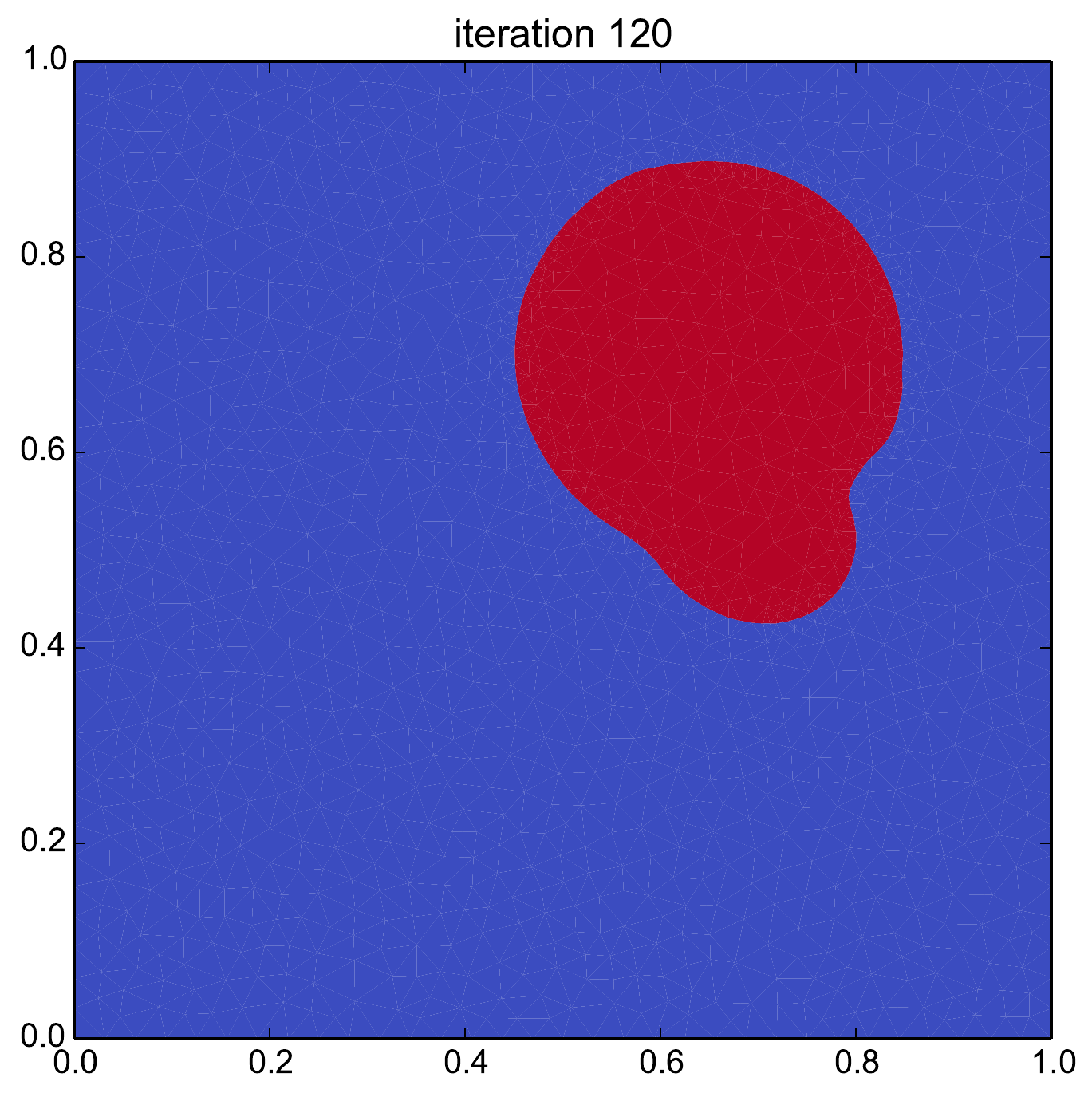}
    \includegraphics[width=0.24\linewidth]{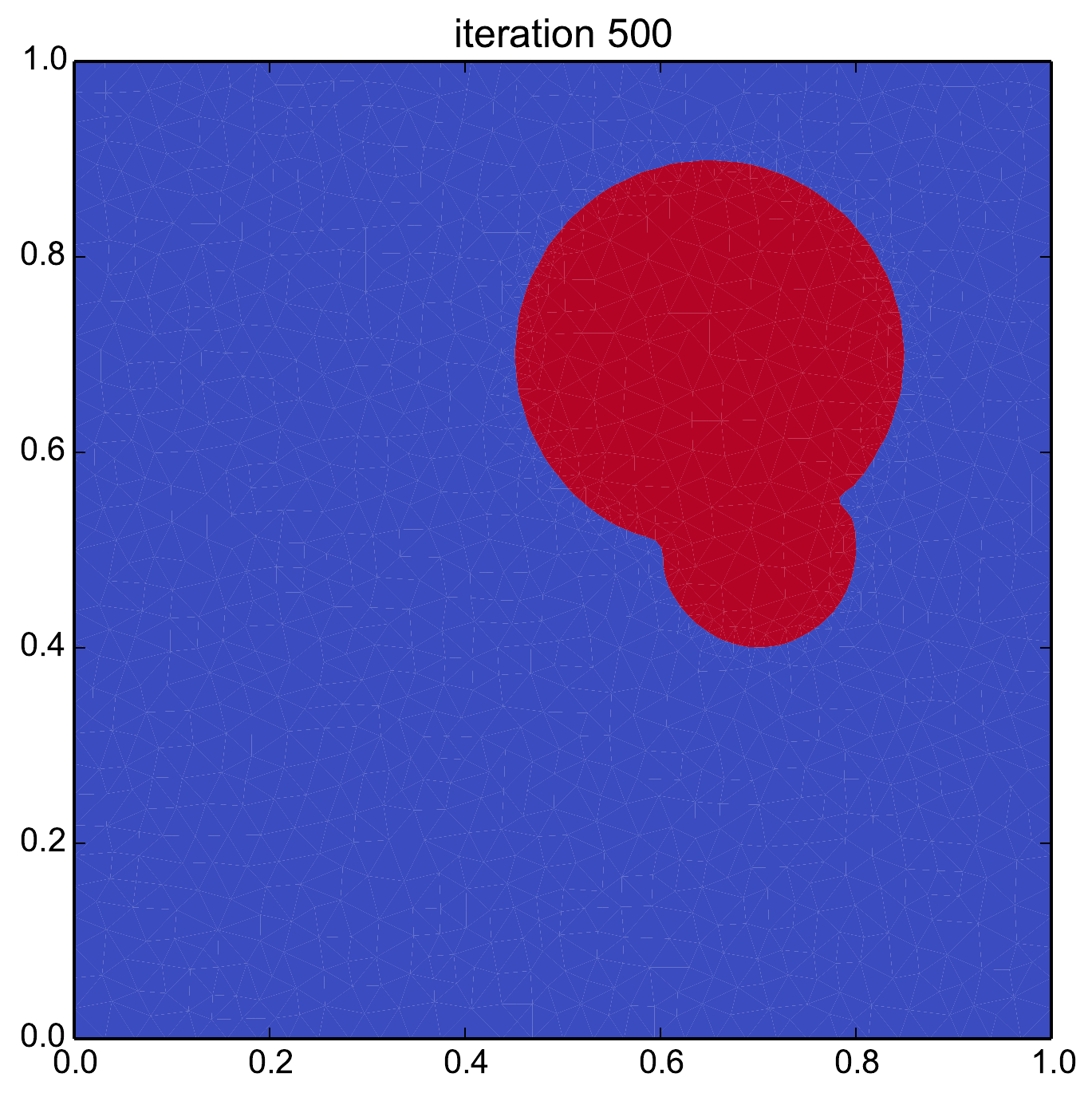}
    \caption{Successive meshes for iteration steps 0, 10, 20, 35, 55, 75, 120, 200, 500 of vvRKHS based optimisation with $\phi_2$.}
    \label{fig:RK2-meshes}
\end{figure}

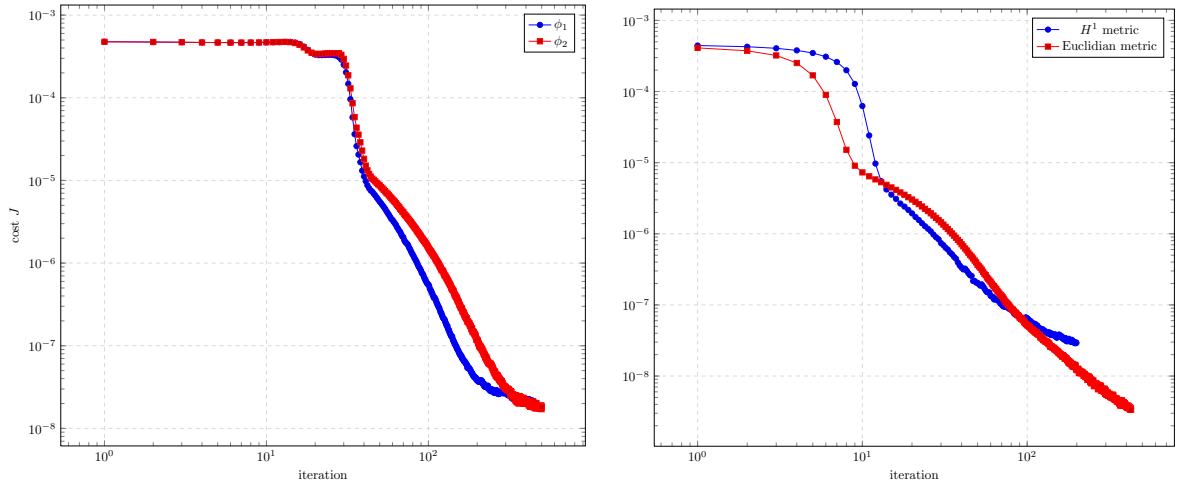
\begin{figure}[h!]
    \begin{center}
        \resizebox{0.50\textwidth}{!}{
            \begin{tikzpicture}
            \pgfplotstableread{data/RK1-RK-1-RKS-10.0-f-1-mN-30-iN-100-it-450-error.dat}\dataA
            \pgfplotstableread{data/RK2-RK-2-RKS-10.0-f-1-mN-30-iN-100-it-500-error.dat}\dataB
            \begin{loglogaxis}[
            width=\linewidth, % scale the plot to \linewidth
            grid=major,
            grid style={dashed,gray!30},
            xlabel=iteration,
            ylabel=cost $J$%,
            %          legend style={at={(0.5,-0.2)},anchor=north}, % put the legend below the plot
            ]
            \addplot table[x=iteration,y=error2] {\dataA}; \addlegendentry{$\phi_1$};
            \addplot table[x=iteration,y=error2] {\dataB}; \addlegendentry{$\phi_2$};
            \end{loglogaxis}
            \end{tikzpicture}
        }%resizebox
        \hfill%
        \resizebox{0.48\textwidth}{!}{
            \begin{tikzpicture}
            \pgfplotstableread{data/pgfB-H1-RK--1-RKS-10.0-f-1-mN-30-iN-100-it-199-error.dat}\dataA
            \pgfplotstableread{data/pgfF3-RK--1-RKS-10.0-f-1-mN-100-iN-50-it-427-error.dat}\dataB
            \begin{loglogaxis}[
            width=\linewidth, % scale the plot to \linewidth
            grid=major,
            grid style={dashed,gray!30},
            xlabel=iteration%,
            %          ylabel=cost $J$%,
            %          legend style={at={(0.5,-0.2)},anchor=north}, % put the legend below the plot
            ]
            \addplot table[x=iteration,y=error2] {\dataA}; \addlegendentry{$H^1$ metric};
            \addplot table[x=iteration,y=error2] {\dataB}; \addlegendentry{Euclidian metric};
            \end{loglogaxis}
            \end{tikzpicture}
        }%resizebox
        \caption{Error progress for vvRKHS based iterations (left), $H^1$ and Euclidian metric (right).}
        \label{fig:errors}
    \end{center}
\end{figure}

\begin{figure}
    \includegraphics[width=0.5\linewidth]{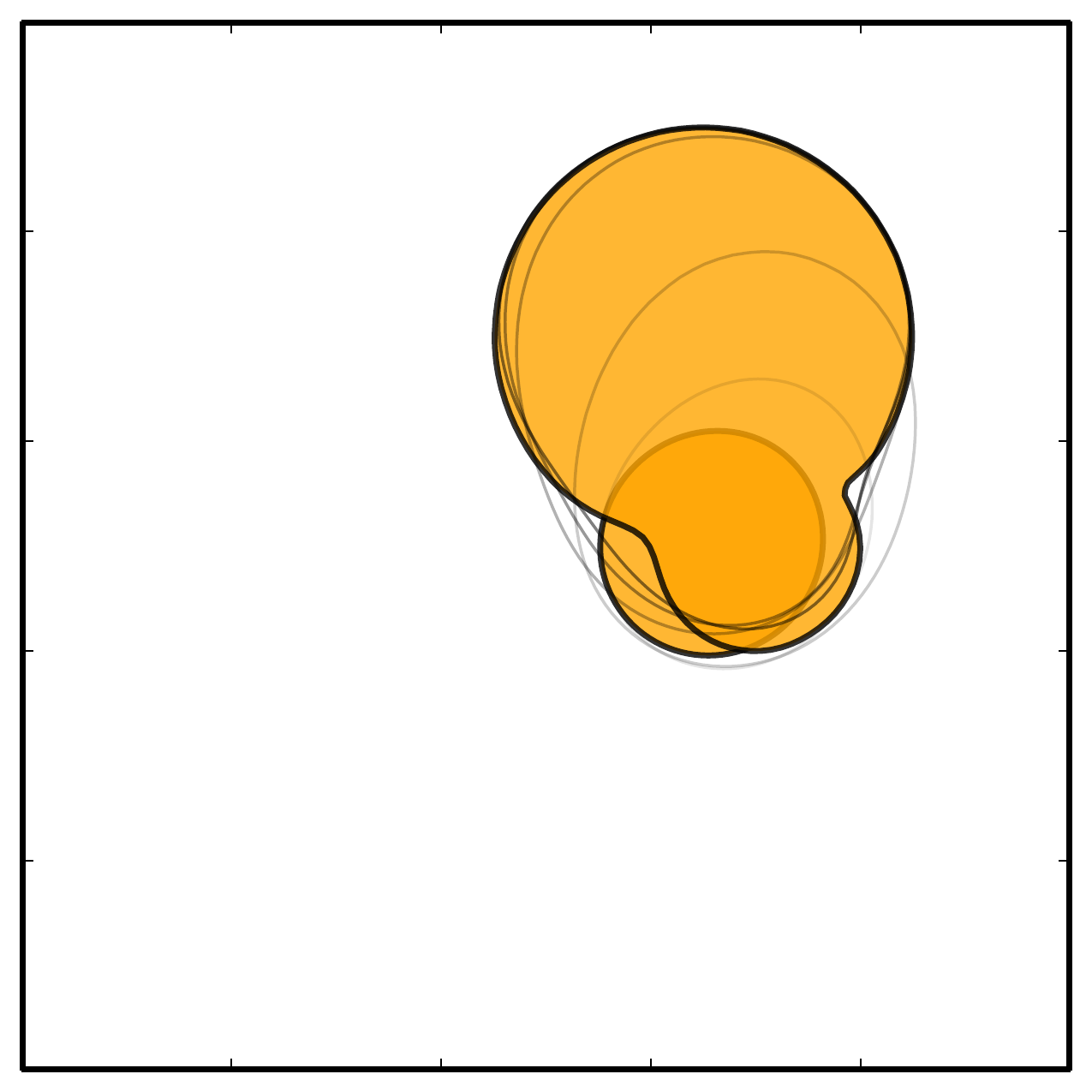}
    \includegraphics[width=0.5\linewidth]{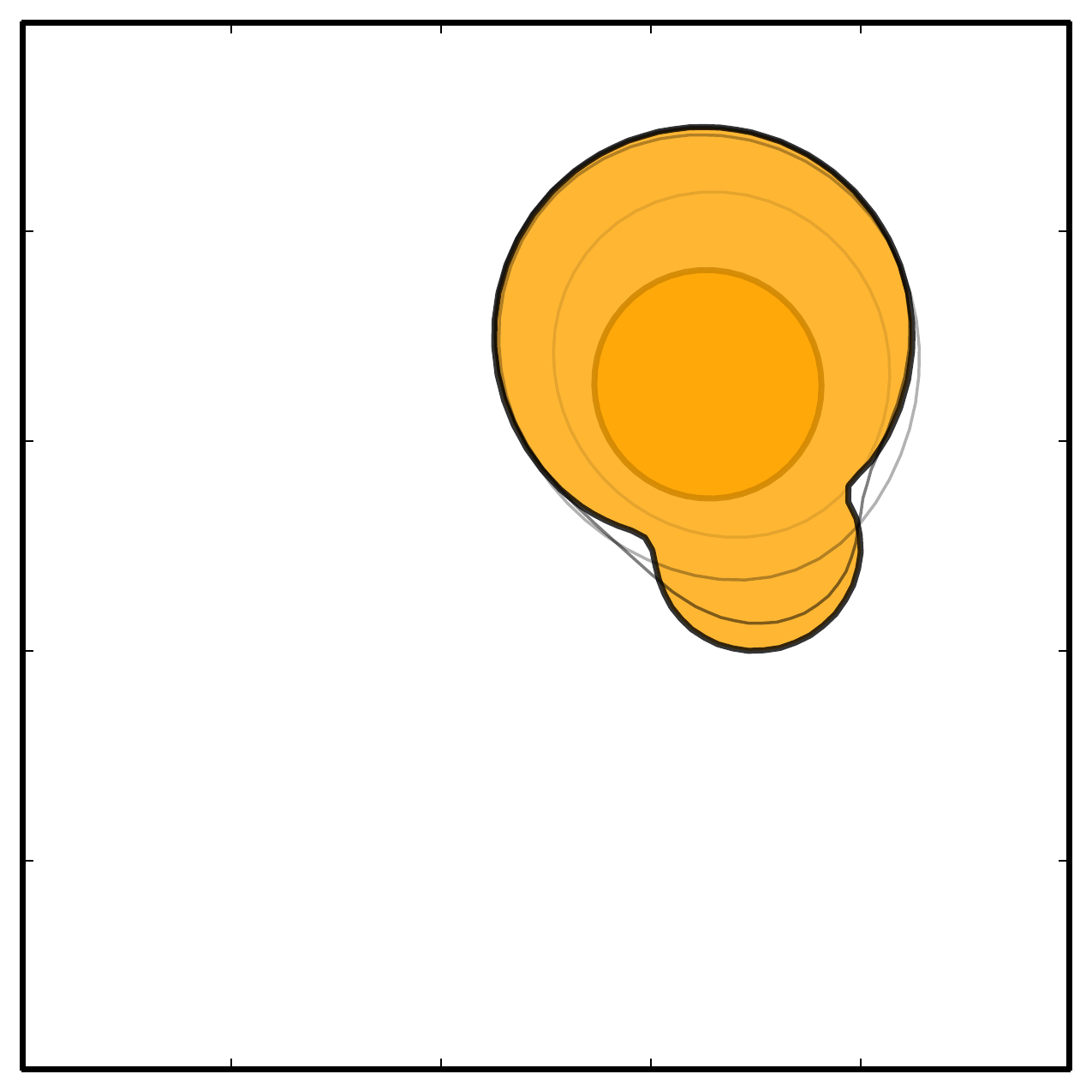}
    \caption{Shape progress for $H^1$ and Euclidian metric based optimisation.}
    \label{fig:H1-shapes}
\end{figure}

\section*{Conclusion}
We examined the applicability of RKHS in PDE constrained shape optimization.
In particular, we showed that many previously used gradient algorithms can be identified as 
methods using gradients computed in RKHS. We also investigated special radial kernels and proposed a 
new variable metrics algorithm which exhibits very promising behaviour in our experimental setting.
A comparison with other common methods shows that our 
method is much more robust when used with more complicated problems, namely when the distance to the optimal shape is
large and its regularity is reduced (examine the non-convex areas).
With the presented derivation and numerical demonstration of the new method, we only scratched the surface
of this promising approach to shape optimisation in RKHS and many highly interesting 
question remain open.
For instance, the ``optimal choice'' of the kernel for specific problems will be subject of future work.  

\bibliographystyle{plain}
\bibliography{references}

\begin{thebibliography}{10}

\bibitem{MR2329288}
L.~Afraites, M.~Dambrine, and D.~Kateb.
\newblock Shape methods for the transmission problem with a single measurement.
\newblock {\em Numer. Funct. Anal. Optim.}, 28(5-6):519--551, 2007.

\bibitem{MR1911658}
G.~Allaire, F.~Jouve, and A.-M. Toader.
\newblock A level-set method for shape optimization.
\newblock {\em C. R. Math. Acad. Sci. Paris}, 334(12):1125--1130, 2002.

\bibitem{MR0051437}
N.~Aronszajn.
\newblock Theory of reproducing kernels.
\newblock {\em Trans. Amer. Math. Soc.}, 68:337--404, 1950.

\bibitem{MR2642680}
M.~Berggren.
\newblock A unified discrete-continuous sensitivity analysis method for shape
  optimization.
\newblock In {\em Applied and numerical partial differential equations},
  volume~15 of {\em Comput. Methods Appl. Sci.}, pages 25--39. Springer, New
  York, 2010.

\bibitem{MR2265340}
C.~Carmeli, E.~De~Vito, and A.~Toigo.
\newblock Vector valued reproducing kernel {H}ilbert spaces of integrable
  functions and {M}ercer theorem.
\newblock {\em Anal. Appl. (Singap.)}, 4(4):377--408, 2006.

\bibitem{MR800331}
M.~Delfour, G.~Payre, and J.-P. Zol{\'e}sio.
\newblock An optimal triangulation for second-order elliptic problems.
\newblock {\em Comput. Methods Appl. Mech. Engrg.}, 50(3):231--261, 1985.

\bibitem{MR2731611}
M.~C. Delfour and J.-P. Zol{\'e}sio.
\newblock {\em Shapes and geometries}, volume~22 of {\em Advances in Design and
  Control}.
\newblock Society for Industrial and Applied Mathematics (SIAM), Philadelphia,
  PA, second edition, 2011.
\newblock Metrics, analysis, differential calculus, and optimization.

\bibitem{MR2299620}
K.~Eppler, H.~Harbrecht, and R.~Schneider.
\newblock On convergence in elliptic shape optimization.
\newblock {\em SIAM J. Control Optim.}, 46(1):61--83 (electronic), 2007.

\bibitem{MR3436555}
P.~Gangl, U.~Langer, A.~Laurain, H.~Meftahi, and K.~Sturm.
\newblock Shape {O}ptimization of an {E}lectric {M}otor {S}ubject to
  {N}onlinear {M}agnetostatics.
\newblock {\em SIAM J. Sci. Comput.}, 37(6):B1002--B1025, 2015.

\bibitem{HenPie05}
A.~Henrot and M.~Pierre.
\newblock {\em Variation et optimisation de formes}, volume~48 of {\em
  Math\'ematiques \& Applications (Berlin) [Mathematics \& Applications]}.
\newblock Springer, Berlin, 2005.
\newblock Une analyse g{\'e}om{\'e}trique. [A geometric analysis].

\bibitem{ls_algo_sensi}
M.~Hinterm\"uller.
\newblock {Fast level-set based algorithms using shape and topological
  sensitivity information}.
\newblock {\em Control and Cybernetics}, 34(1):305--324, 2005.

\bibitem{MR3348199}
R.~Hiptmair, A.~Paganini, and S.~Sargheini.
\newblock Comparison of approximate shape gradients.
\newblock {\em BIT}, 55(2):459--485, 2015.

\bibitem{MR2887931}
A.~Laurain and Y.~Privat.
\newblock On a {B}ernoulli problem with geometric constraints.
\newblock {\em ESAIM Control Optim. Calc. Var.}, 18(1):157--180, 2012.

\bibitem{lauraindistributed}
A.~Laurain and K.~Sturm.
\newblock Distributed shape derivative via averaged adjoint method and
  applications.
\newblock {\em accepted for publication in ESAIM:M2AN}, 2015.

\bibitem{MR2201275}
P.~W. Michor and D.~Mumford.
\newblock Riemannian geometries on spaces of plane curves.
\newblock {\em J. Eur. Math. Soc. (JEMS)}, 8(1):1--48, 2006.

\bibitem{MR1939127}
S.~Osher and R.~Fedkiw.
\newblock {\em Level set methods and dynamic implicit surfaces}, volume 153 of
  {\em Applied Mathematical Sciences}.
\newblock Springer-Verlag, New York, 2003.

\bibitem{P14_562}
A.~Paganini.
\newblock Approximate shape gradients for interface problems.
\newblock Technical Report 2014-12, Seminar for Applied Mathematics, ETH
  Z{\"u}rich, Switzerland, 2014.

\bibitem{hiptpagan15}
A.~Paganini and R.~Hiptmair.
\newblock Approximate riesz representatives of shape gradients.
\newblock {\em Seminar for Applied Mathematics, ETH Zurich (Technical Report)}.

\bibitem{SokoZol}
J.~Soko{\l}owski and J.-P. Zol{\'e}sio.
\newblock {\em Introduction to shape optimization}, volume~16 of {\em Springer
  Series in Computational Mathematics}.
\newblock Springer-Verlag, Berlin, 1992.
\newblock Shape sensitivity analysis.

\bibitem{MR3374631}
K.~Sturm.
\newblock Minimax {L}agrangian approach to the differentiability of nonlinear
  {PDE} constrained shape functions without saddle point assumption.
\newblock {\em SIAM J. Control Optim.}, 53(4):2017--2039, 2015.

\bibitem{sturm2015shape}
K.~Sturm.
\newblock {\em On shape optimization with non-linear partial differential
  equations}.
\newblock PhD thesis, Berlin, Technische Universit{\"a}t Berlin, Diss., 2015.

\bibitem{SturmHoemHint13}
K.~Sturm, D.~H\"omberg, and M.~Hinterm\"uller.
\newblock Distortion compensation as a shape optimisation problem for a sharp
  interface model.
\newblock {\em Comp. Optim. and Appl.}, pages 1--32, 2016.

\bibitem{MR2131724}
H.~Wendland.
\newblock {\em Scattered data approximation}, volume~17 of {\em Cambridge
  Monographs on Applied and Computational Mathematics}.
\newblock Cambridge University Press, Cambridge, 2005.

\bibitem{MR2656312}
L.~Younes.
\newblock {\em Shapes and diffeomorphisms}, volume 171 of {\em Applied
  Mathematical Sciences}.
\newblock Springer-Verlag, Berlin, 2010.

\end{thebibliography}

\end{document}